\documentclass[a4paper,8pt]{article}
\usepackage[francais]{babel}
\usepackage[T1]{fontenc}
\usepackage{geometry}
\usepackage[latin1]{inputenc}
\usepackage{amsmath,amssymb,mathrsfs,amsthm}
\usepackage{soul}
\usepackage{epsfig}
\usepackage{hyperref}
\usepackage{graphicx}
\usepackage{xypic}
\usepackage{datetime}
\usepackage{fancyhdr}
\fancypagestyle{plain}{
\lhead{}
\rhead{}
\lfoot{Compil\'e le \today{} \`a \currenttime}

}

\newtheorem{theorem}{Th\'eor\`eme}[section]
\newtheorem{proposition}[theorem]{Proposition}
\newtheorem{lemma}[theorem]{Lemme}
\newtheorem{Corollaire}[theorem]{Corollaire}
\newtheorem{definition}[theorem]{D\'efinition}
\newtheorem{example}[theorem]{Exemple}
\newtheorem{remarque}[theorem]{Remarque}

\newcommand{\vc}{\|\cdot\|}
\newcommand{\C}{\mathbb{C}}
\newcommand{\mc}{\mathcal{O}(1)}

\newcommand{\Z}{\mathbb{Z}}

\newcommand{\lra}{\longrightarrow}
\newcommand{\cl}{\mathcal{C}^\infty}
\newcommand{\al}{\alpha}
\newcommand{\A}{\mathcal{A}}
\newcommand{\R}{\mathbb{R}}
\newcommand{\p}{\mathbb{P}}
\newcommand{\eps}{\varepsilon}
\newcommand{\X}{\mathcal{X}}
\newcommand{\T}{\mathbb{T}}

\newcommand{\vf}{\varphi}
\newcommand{\si}{\sigma}
\newcommand{\Q}{\mathbb{Q}}
\newcommand{\N}{\mathbb{N}}
\newcommand{\la}{\lambda}
\newcommand{\ms}{\mathcal{S}}
\title{Hauteurs canoniques des sous-vari\'et\'es toriques }
\date{\today, \currenttime}
\author{Mounir Hajli\footnote{\small National Center for Theoretical Sciences, Taipei Office, National Taiwan University,
Taipei 106, Taiwan \quad\quad\quad \emph{E-mail}:\ttfamily{ hajlimounir@gmail.com}}}
\begin{document}

\maketitle

\begin{abstract}
On pr\'esente une formule explicite pour les hauteurs canoniques pour
une
classe de sous-vari\'et\'es toriques projectives
au  sens de Gelfand, Kapranov
et Zelevinsky. Notre approche donne une alternative partielle 
aux calculs  de \cite{PS}.
\end{abstract}

\section{Introduction}
La hauteur d'une vari\'et\'e arithm\'etique est un analogue arithm\'etique du degr\'e g\'eom\'etrique,
 mesurant la complexit\'e de la vari\'et\'e. Il est connu que le calcul explicite des hauteurs
  est une
 t\^ache tr\`es difficile et compliqu\'ee. Il y a peu d'exemples de calculs d'hauteur d'une vari\'et\'e.
  Le cas torique respr\'esente  une situation particuli\`ere et int\'eressante.
L'arithm\'etique  des vari\'et\'es toriques a \'et\'e  etudi\'ee
de mani\`ere intense par plusieurs
 auteurs. Dans \cite{PS}, Philippon et Sombra pr\'esentent
une expression explicite pour la hauteur normalis\'ee du translat\'e d'une vari\'et\'e torique projective d\'efinie sur
$\overline{\Q}$, cf. \cite[Th\'eor\`eme 0.1]{PS}. Cette expression se d\'ecompose comme somme de contributions
locales, chaque terme \'etant l'int\'egrale d'une certaine fonction concave et affine par morceaux, d\'efinie
sur le polytope $Q_\mathcal{A}$. Leur d\'emarche pour d\'emonter leur formule  s'appuie
sur le calcul d'une  fonction de Hilbert arithm\'etique appropri\'ee,  au lieu d'utiliser la d\'efinition de la
hauteur normalis\'ee. Dans \cite{Burgos2},  les auteurs \'etablissent une formule int\'egrale pour la hauteur
 d'une vari\'et\'e torique projective par rapport \`a un fibr\'e en droites \'equivariant muni d'une m\'etrique
 hermitienne, continue et invariante par l'action du tore compact de la vari\'et\'e (voir \cite[Theorem 5.1.4]{Burgos2}).
 Ils  d\'eduisent (voir \cite[Corollary 5.2.6]{Burgos2}) une nouvelle preuve  pour le r\'esultat de Philippon et Sombra \cite{PS}.

 Rappelons le r\'esultat majeur de 
\cite{PS}. Soit $\p^n$ l'espace projectif 
sur $\overline{\Q}$
 de dimension $n$. Soit $\A=\{a_0,a_1,\ldots,a_n\}$ une
famille  de vecteurs de $\Z^d$. Soit $L_\A$ le sous-module engendr\'e
par les diff\'erences des vecteurs $a_0,\ldots,a_n$. Soit 
$\beta\in (\overline{\Q}^\ast)^{n+1}$. On note par $X_{\A,\beta}$
la vari\'et\'e torique au sens de Gelfand, Kapranov et Zelevinsky
 associ\'ee \`a $\A$ (voir les rappels de la section 
 \eqref{x14}). Soit
$K$ un corps de nombres appropri\'e tel que 
$\beta\in (K^\ast)^{n+1}$. Philippon et Sombra 
associent \`a $X_{\A,\beta}$ une famille de fonctions $(\vartheta_{\A,\beta,v})$
index\'ee par $M_K$, l'ensemble des places du corps $K$, o\`u chaque
fonction $\vartheta_{\A,\beta,v}$ est donn\'ee comme suit
\[
\vartheta_{\A,\beta,v}:Q_\A\rightarrow \R,\quad x\mapsto 
\max\{ y\in \R|\, (x,y)\in Q_{\A,\beta,v}\},
\]
avec $Q_{\A,\beta,v}:=\mathrm{Conv}\bigl( (a_0,\log
|\beta_0|_v),\ldots,(a_n,\log|\beta_n|_v)\bigr)\subset \R^{d+1}$ et
$Q_{\A}:=\mathrm{Conv}(a_0,\ldots,a_n)\subset \R^d$. Notons que 
que la d\'etermination de $\vartheta_{\A,\beta,v}(x)$ pour 
$x\in Q_{\A}$ est un probl\`eme d'optimisation lin\'eaire. En effet, 
trouver $\vartheta_{\A,\beta,v}(x)$ est \'equivalent  \`a r\'esoudre  
le probl\`eme  suivant

\begin{equation}\label{x15}
\begin{cases}
\text{Minimiser}\, <\la,c> &\\
\text{sous les contraintes }B\la= b &\\
\text{et}\,\la\geq 0 &
\end{cases}
\end{equation}
o\`u $c=(-\log |\beta_0|_v,\ldots,-\log|\beta_n|_v)$,
$B$ est une matrice de taille $(d+1)\times n$ et $b$ un vecteur
de $\R^{d+1}$ qui s'\'ecrivent en fonction des $a_0,\ldots,a_n$, 
$(1,\ldots,1)$ et de $x$. La th\'eorie de l'optimisation
lin\'eaire affirme que si $\la^\ast$ est une solution optimale
de  \eqref{x15}, alors  $\la^\ast$ appartient 
au bord d'un polytope convexe  d\'efini par les contraintes du 
probl\`eme en question.  Il existe
plusieurs algorithmes de r\'esolution de \eqref{x15}, par exemple
la m\'ethode du Simplexe.  Philippon et Sombra dans \cite{PS} donnent
une formule pour $\widehat{h}(X_{\A,\beta})$, la 
hauteur normalis\'ee de $X_{\A,\beta}$ en fonction  des 
$\vartheta_{\A,\beta,v}$ pour $v\in M_K$. Plus pr\'ecis\'ement, ils
\'etablissent le r\'esultat suivant
\[
\widehat{h}(X_{\A,\beta})=(d+1)!\sum_{v\in M_K}\frac{[K_v:\Q_v]}{[K:\Q]}
\int_{Q_\A}\vartheta_{\A,\beta,v}dx_1\cdots dx_d.
\]
voir \cite[Th\'eor\`eme 3.6]{PS}.

Le r\'esultat de Philippon et Sombra s'applique donc
\`a toute sous-vari\'et\'e torique $X_{\A,\beta}$ 
translat\'ee par $\beta\in (\overline{\Q}^\ast)^{n+1}$ quelconque. En observant que les fonctions 
$\vartheta_{\A,\beta,v}$ sont affines  par morceaux, alors
on peut expliciter les int\'egrales dans la formule
de la hauteur en utilisant les techniques de la 
programmation 
lin\'eaire. Notons que la m\'ethode de \cite{PS} 
a \'et\'e  g\'en\'eralis\'ee aux fibr\'es en droites
hermitiens semipositifs, voir \cite{Burgos3} pour plus
 de d\'etails. \\

Dans cet article, nous proposons  une alternative partielle 
\`a l'approche de \cite{PS} pour le calcul des hauteurs
canoniques pour une classe de vari\'et\'es toriques projectives. Notre m\'ethode
s'appuie essentiellement sur la d\'efinition de la hauteur canonique, en 
plus elle fournit  un
moyen pour le calcul explicite de la hauteur.\\

Soient $K $  un corps de nombres et $\mathcal{O}_K$ son anneau des entiers. Soit $\mathcal{X}$ une vari\'et\'e
arithm\'etique (projective) de dimension  $n$ sur  $ \ms:=\mathrm{Spec}(\mathcal{O}_K)$ et note
par $\pi$ le morphisme structural. On renvoie \`a
\cite{AIT}, ainsi qu'\`a
\cite{BoGS} pour la construction des groupes de Chow arithm\'etiques. On rappelle (voir par exemple
\cite[\S 2.1.3]{BoGS}) que l'on dispose de deux morphismes:
\[
  \deg:\widehat{CH}^0(\ms)=CH^0(\ms)\lra \Z, \quad
 \widehat{\deg}:\widehat{CH}^1(\ms)\lra \R,
 \]
qui induisent par composition avec $\pi_\ast:\widehat{CH}^\ast(\X)\lra
\widehat{CH}^{\ast-n}(\ms)$ les morphismes:
\[
\begin{split}
  \deg:&{CH}^n(\X)\lra \Z,\quad \text{(degr\'e g\'eom\'etrique)},\\
  \widehat{\deg}&:\widehat{CH}^{n+1}(\X)\lra \R\quad \text{(degr\'e arithm\'etique)}.
\end{split}
\]
On consid\`ere $\p^n_{\mathcal{O}_K}=\mathrm{Proj}\bigl(\mathcal{O}_K[T_0,\ldots,T_n]\bigr)$
l'espace projectif de dimension $n$ sur $\mathrm{Spec}(\mathcal{O}_K)$. Soit $Z$ un cycle sur
$\mathbb{P}^n_{\mathcal{O}_K}$,  on note par
 $h_{\overline{\mathcal{O}(1)}_\infty}(Z)$ la  hauteur canonique de $Z$, o\`u $\overline{\mathcal{O}(1)}_\infty$ est  le fibr\'e
universel sur $\p^n_{\mathcal{O}_K}$  muni de sa m\'etrique canonique (voir \cite[Proposition-d\'efinition
5.5.1]{Maillot} ou \cite{Zhang}).
\subsection{R\'esultats et strat\'egie de la preuve}

Soit $\beta=(\beta_0,\ldots,\beta_n)\in
(\overline{\Q}^\ast)^{n+1}$.
 Soient
$d\in \N_{\geq 1}$ et
 $\A=\{a_1,\ldots,a_n \}$
  une sous-famille  de vecteurs de $\Z^d$ de cardinal $n$ tels que
 $L_{\A}:=\Z a_1+\cdots+\Z a_n=\Z^d$.
  On consid\`ere $X_{\A,\beta}$ la vari\'et\'e 
 torique associ\'ee (voir rappels de la section\eqref{x14}).  Soit $K$ un corps de nombres de d\'efinition de $X_{\A,\beta}$. La donn\'ee $\A$ d\'efinit $n-d$ polyn\^omes homog\`enes $R_1,\ldots,R_{n-d}$. Dans cet article, on
  suppose que 
 $X_{\A,\beta}$ satisfait les hypoth\`eses suivantes:  
 \begin{enumerate}
 \item Les polyn\^omes $R_1,\ldots,R_{n-d}$ d\'efinissent $X_{\A,\beta}$.
 \item $Z(R_k,R_{k+1},\ldots,R_{n-d})$, la vari\'et\'e projective 
 d\'efinie par $R_k,\ldots,R_{n-d}$, est int\`egre pour tout
 $k=1,\ldots,n-d$.
 \end{enumerate}
 En particulier, $X_{\A,\beta}$ est une intersection compl\`ete. Voir  \eqref{exemple2}
 pour des exemples. 
Pour simplifier, on dira que $X_{\A,\beta}$ satisfait l'hypoth\`ese
$ \mathscr{A}$, si $(1.)$ et $(2.)$ sont v\'erifi\'es.   Aussi, on suppose qu'il existe $F$, un corps de nombres
tel que son anneau des entiers $\mathcal{O}_F$ soit un anneau factoriel et 
que 
$\beta\in (F^\ast)^{n+1}$. Sous cette hypoth\`ese, on v\'erifie (voir Lemme \ref{x5})  que pour 
tout $i\in 
 \{1,\ldots,n-d \}$, il existe $\tau_{i}\in K$ tel que
 \[
 (R_i)\cap \mathcal{O}_K[T_0,\ldots,T_n]=(\tau_{i} R_i). \]
 
 Notons que  la condition
 $\mathscr{A}$ ne d\'epend pas de $\beta$ (voir Remarque 
 \ref{nedepenbeta}). \\

On note par $\X_{\A,\beta}$ la cl\^oture de Zariski de $X_{\A,\beta}$
dans $\p^n_{\mathcal{O}_K}$.
On se propose de donner une expression explicite pour 
$h_{\overline{\mc}_\infty}\bigl(\X_{\A,\beta}\bigr)$ en fonction de $\beta$ et de $\A$. La condition $\mathscr{A}$ combin\'ee
avec Lemme \ref{lemmemodele}  nous   
permet  de construire $\X_{\A_1,
 \beta}:=\X_{\A,\beta},\X_{\A_2,\beta},
 \ldots,\X_{\A_{n-d},\beta}$ des
sous-sch\'emas toriques int\`egres de $\p^n_{\mathcal{O}_K}$ avec

\begin{equation}\label{x11}
\X_{\A,\beta}=:\X_{\A_1, \beta}\subsetneq \X_{\A_2,\beta}\subsetneq \cdots\subsetneq \X_{\A_{n-d},\beta}\subsetneq 
\p^n_{\mathcal{O}_K},
\end{equation}
et tels que chaque $\X_{\A_i,\beta}$ soit une hypersurface dans 
$\X_{A_{i+1},\beta}$ donn\'ee par $s_i:=\tau_{i} R_{i}$ 
 pour tout  $i=1,\ldots,n-d-1$ et $X_{\A_{n-d},\beta}$ est une hypersurface de $\p^n_{\mathcal{O}_K}$ donn\'ee par $s_{n-d}=\tau_{n-d}R_{n-d}$, o\`u  $d$ 
 est la dimension de $\X_{\A,\beta}$.

Par la formule de Faltings (voir \cite[Th\'eor\`eme 5.5.6]{Maillot})  on a pour tout $j=0,1,\ldots,n-d-1$

\begin{equation}\label{x10}
\begin{split}
h_{{}_{\overline{\mc}_\infty}}\bigl(\X_{\A_j,\beta}\bigr)&=
\deg(R_{
j})h_{{}_{\overline{\mc}_\infty}}\bigl(\X_{\A_{j+1
},\beta}\bigr)+\sum_{\si:K\rightarrow \C}\int_{\X_{\A_{j+1},\beta}(\C)}\log\bigl\|s_j\bigr\|_{\si,\infty}c_1\bigl(
\overline{\mc}_\infty\bigr)^{j+d},\\
h_{{}_{\overline{\mc}_\infty}}\bigl(\X_{\A_{n-d},\beta}\bigr)&=
\deg(R_{n
-d})h_{{}_{\overline{\mc}_\infty}}\bigl(\p^n_{\mathcal{O}_K}\bigr)+\sum_{\si:K\rightarrow \C}\int_{\p^n(\C)}\log\bigl
\|s_j\|_{\si,\infty} c_1\bigl(
\overline{\mc}_\infty\bigr)^{n}
\end{split}
\end{equation}

Sachant que
$h_{\overline{\mc}_\infty}\bigl(\X_{\A_{n-d},\beta}
\bigr)=h_{\overline{\mc}_\infty}\bigl(\p^n_{\mathcal{O}_K}
\bigr)=0$ (voir \cite[Proposition 7.1]
{Maillot}).
Alors, d'apr\`es la formule pr\'ec\'edente, le calcul de 
$h_{\overline{\mc}_\infty}\bigl(\X_{\A_j,\beta}\bigr)$ se d\'eduira 
 de celui de $h_{\overline{\mc}_\infty}\bigl(\X_{\A_{j+1},\beta}\bigr)$ si 
 l'on d\'etermine
l'expression de \[
                 c_1\bigl(
\overline{\mc}_\infty\bigr)^{j+d+1}_{|_{\X_{\A_{j+1},\beta}}}\quad  \forall j\in \N,
                \]
 C'est l'objet du th\'eor\`eme ci-dessous. Si 
 $\A'=\{a_1',\ldots,a_n'\}$
 est une famille de vecteurs de $\Z^{d+1}$ qui engendre  $\Z^{d+1}$, nous
montrons qu'il existe $S$ un ensemble fini  de points de $\R^{d+1}$ et un ensemble
$\bigl(X_{\A'_{I_s},1}\bigr)_{s\in
S}$ de vari\'et\'es toriques d\'etermin\'es explicitement tels que:

\begin{theorem}[voir Th\'eor\`eme \ref{courant}]
Sur $\T^{d+1}(\C)$, on a l'\'egalit\'e de courants suivante:
\[
\omega_{\beta}^{d+1}:=(\ast_{\A',\beta})^\ast\Bigl( c_1\bigl(
\overline{\mc}_\infty\bigr)^{j+d+1}_{|_{\X_{\A_{j+1},\beta}}}\Bigr)=\bigl(dd^c\log \max(|\beta\cdot t^{a'}|) \bigr)^{d+1}=\sum_{s\in S}
\deg(X_{\A'_{I_s},1})
\delta_{{}_{\mathbf{S}_s}},
\]
avec $\beta=(1,\beta_1,\ldots,\beta_n)\in (\overline{\Q}^\ast)^{n+1}$ 
et $|\beta\cdot t^{a'}|:=\bigl(1,|\beta_1 t^{a_1'}|,\ldots,|\beta_n t^{a_n'}|\bigr)$ pour
tout $t\in \T^{d+1}(\C)$ et $\delta_{\mathbf{S}_s} $ est le courant int\'egration sur  $\mathbf{S}_s$ (voir
\eqref{cerlce}).
\end{theorem}

\begin{remarque} 
\rm{Notons que la preuve du  th\'eor\`eme ci-dessus 
 permet de trouver explicitement 
l'ensemble $S$ et les
 coefficients $\deg(X_{\A'_{I_s},1})$.}
\end{remarque}

En combinant l'\'equation \ref{x10} et le th\'eor\`eme pr\'ec\'edent 
 nous pouvons  donc 
 d\'eterminer  par r\'ecurrence  
la hauteur canonique des vari\'et\'es toriques. Plus concr\`etement, on a
le r\'esultat suivant

\begin{theorem}[voir Th\'eor\`eme \ref{laformulehot}]Soit $n\in \N^\ast$ et $d\in 
\{1,\ldots,n-1\}$. Soit $\A:=\{a_1,\ldots,a_n\}$ une
sous-famille de $\Z^d$ de rang $d$. Soit $\beta\in 
(\overline{\Q}^\ast)^{n+1}$. On suppose que $X_{\A,\beta}$  v\'erifie l'hypoth\`ese 
$\mathscr{A}$ et que $\beta\in (F^\ast)^{n+1}$ o\`u $F$ est un corps 
de nombres tel que son anneau des entiers est factoriel. On a,  
\begin{itemize}
\item Si $d\leq n-1$. On note $\A':=\A_2$ (voir \ref{x3})  alors
il  existe $\tau\in K$
tel que
\begin{align*}
h_{\overline{\mc}_\infty}\bigl(\X_{\A,
\beta}\bigr)=&\deg(R_{1})\,h_{\overline{\mc}_\infty}\bigl(\X_{\A',
\beta}\bigr)+\deg(X_{\A',1})\log|N_K(\tau)|\\
&+\sum_{\si:K\rightarrow \C}\sum_{s_\si\in S_\si}
\deg(X_{\A'_{I_{s_\si}},1}) \int_{t\in\T^{d+1}(\C)}\log \frac{|Q_{1}(\beta
\cdot t^{a'}) |_\si}{\max(|\beta\cdot t^{a'} |_\si 
)^{\deg(R_1)} }
\delta_{\mathbf{S}_{s_\si}}.
\end{align*}
\item  Si $d=n-1$, alors il existe aussi $\tau\in K$ tel que
\begin{align*}
h_{\overline{\mc}_\infty}\bigl(\X_{\A,
\beta}\bigr)=\log|N_K(\tau)|
+\sum_{\si:K\rightarrow \C} \int_{t\in (\mathbb{S}^1)^{n}}\log |Q_{1}(t_1,\ldots,t_n
 ) |_\si\frac{dt_1\wedge dt_n}{t_1\cdots t_n}.
\end{align*}
\end{itemize}
\end{theorem}

\begin{remarque}\rm{
Notons que le terme int\'egrale dans notre formule peut \^etre 
simplifier (voir Remarque \ref{x7}). Plus 
 concr\`etement, on donne pour tout $j=1,\ldots,n-d-1$ une formule 
pour 
\[
h_{\overline{\mc}_\infty}\bigl(\X_{\A_j,
\beta}\bigr)-\deg(R_{j})h_{\overline{\mc}_\infty}\bigl(\X_{\A_{j+1},
\beta}\bigr)
\]
et 
\[
h_{\overline{\mc}_\infty}\bigl(\X_{\A_{n-d},
\beta}\bigr),
\]
en fonction de $\beta$, $[K:\Q]$, $\tau_1,\ldots,\tau_{n-d}$,  
$w_{1},\ldots,w_{n-d}$ et les \'el\'ements de $\A_1,\ldots,\A_{n-d}$ (voir \ref{x18} et
\ref{ee1}).  Ainsi, on dispose d'une formule
pour $h_{\overline{\mc}_\infty}\bigl(\X_{\A,
\beta}\bigr)$ en fonction de  $\beta$, $\deg(R_1),\ldots,
\deg(R_{n-d})$, $[K:Q]$, $\tau_1,\ldots,
\tau_{n-d}$,  
$w_1,\ldots, w_{n-d}$ et les \'el\'ements de $\A_{1},\ldots,
\A_{n-d}$. }
\end{remarque}
 
\begin{Corollaire}(voir Corollaire \eqref{x6}) Soit 
$\beta=(1,\beta_1,\ldots,\beta_n)\in (\overline{\Q}^\ast)^{n+1}$. 
En gardant les m\^emes hypoth\`eses que 
dans \eqref{laformulehot}, alors
il existe $u_{\A}\in \N^{n-d}$ et $(v_{\A,\si,i})_{
\substack{\si:K\rightarrow \C\\
i=1,\ldots,n}}$ une sous famille de $\Q^n$ tels que
\[
h_{\overline{\mc}_\infty}\bigl(\X_{\A,
\beta}\bigr)=\sum_{i=1}^{n-d}u_{\A,i}\log|N_K(\tau_i)|+                                                                                                                                                                                                                                                                                                                                                                                                                                                                                                                                                                                                                                                                                                                                                                                                                                                                                                                                                                                                                                                                                                                                                                                                                                                                                                                                                                                                                                                                                                                                                                                                                                                                                                                                                                                                                                                                                                                                                                                                                                                                                                                                                                 
\sum_{\substack{\si:K\rightarrow \C\\i=1,\ldots,n}}v_{\A,\si,i}\log
|\beta_i|_\si.
\]
On a,
\[
h_{\overline{\mc}_\infty}\bigl(\X_{\A,
\beta}\bigr)\in \log\bigl(\overline{\Q}\cap \R_{>0} \bigr).
\]
Donc, si $h_{\overline{\mc}_\infty}\bigl(\X_{\A, \beta}\bigr)\neq 0$,  alors c'est un nombre transcendant.

\end{Corollaire}

Dans Th\'eor\`eme \eqref{xample} on donne un exemple de 
calcul d'hauteur canonique. On consid\`ere la courbe  torique 
$\X_{\A,c}\subset \p^3_{\Q}$ o\`u $\A=\{1,-1,3 \}$ et 
$c=(1,c_1,c_2,c_3)\in (\Q^\ast)^{4}$. On v\'erifiera que cette vari\'et\'e
 satisfait l'hypot\`ese $\mathscr{A}$ et   que ici
 $F=\Q$ dont  l'anneau des entiers est bien \'evidemment factoriel.  En particulier, si $c_1,c_2,c_3\in \Z$ avec 
 $\gcd(c_i,c_j)=1$ pour tout $i\neq j\in \{1,2,3\}$ alors, on a
 \[
 h_{{}_{\overline{\mathcal{O}(1)}_\infty}}(\X_{\A,c})=2\log
 |c_1c_2|+\log\max(|c_1^2|,|c_2c_3|).
 \]

On peut se demander si notre approche peut \^etre g\'en\'eraliser
\`a des vari\'et\'es toriques qui ne v\'erifient pas n\'ecessairement  les hypoth\`eses pr\'ec\'edentes. On propose
alors une g\'en\'eralisation partielle qui consiste \`a \'etudier
la hauteur canonique  des vari\'et\'es toriques qui sont image
d'une vari\'et\'e $X_{A,\beta}$ satisfaisant les hypoth\`eses 
pr\'ec\'edentes,
par un morphisme monomial $\p^n\rightarrow \p^N$ (voir Remarque
 \eqref{x300}). \\

\noindent{\textbf{Remerciements}}: Ce travail a \'et\'e r\'ealis\'e durant ma th\`ese sous la direction de Vincent Maillot, je tiens \`a le remercier  pour ses conseils  lors
de la r\'edaction de cet article. Je remercie Antoine Chambert-Loir pour m'avoir signaler une erreur dans une version 
ant\'erieure.  Je remercie vivement   le referee 
pour ses remarques et corrections pertinentes.

\tableofcontents

\section{La Construction de Gelfand, Kapranov et Zelevinsky}\label{x14}

On note par $\overline{\Q}$ la cl\^oture alg\'ebrique de $\Q$, le corps des nombres rationnels.\\

Soit $\mathbb{T}^d=(\overline{\mathbb{Q}}^\ast)^d$ le tore alg\'ebrique et $\mathbb{P}^n(\overline{\Q})$ l'espace projectif sur $\overline{\mathbb{Q}}$, de
dimension $d$ et $n$ respectivement. Soit $\mathcal{A}=\{a_0,\ldots,a_n\}$ une suite de $n+1$ vecteurs de $\mathbb{Z}^d$. L'ensemble
 $\A$ d\'efinit une
action  de $\mathbb{T}^d$ sur $\mathbb{P}^n(\overline{\Q})$ comme suit:
\begin{align*}
\ast_\mathcal{A}:\mathbb{T}^d\times \mathbb{P}^n(\overline{\Q})&\longrightarrow \mathbb{P}^n(\overline{\Q})\\
(s,x)&\longrightarrow \bigl[s^{a_0}x_0:\cdots:s^{a_n}x_n\bigr],
\end{align*}
o\`u $s^{a_i}:=\prod_{j=1}^ds^{a_{ij}}$ pour tout $i$.

Soit $X_{\mathcal{A},\alpha}$ l'adh\'erence de Zariski de l'image de l'application monomiale:

\begin{equation}\label{tor}
\ast_{\mathcal{A},\alpha}:=\ast_\mathcal{A}|_{\alpha}:\mathbb{T}^d\longrightarrow \mathbb{P}^n(\overline{\Q}), \quad s\rightarrow
[\alpha_0s^{a_0}:\cdots:\alpha_ns^{a_n}]
\end{equation}

C'est une vari\'et\'e torique projective au sens de Gelfand, Kapranov et Zelevinsky \cite{GKZ},
c'est \`a dire une sous-vari\'et\'e de $\mathbb{P}^n$ stable par rapport \`a l'action de $\mathbb{T}^d$,
avec une
orbite dense $X_{\mathcal{A},\alpha}^\circ:=\mathbb{T}^d\ast_\mathcal{A}\alpha$.\\

Dans la suite on suppose que $a_0=0$ (puisque
$[\alpha_0s^{a_0}:\cdots:\alpha_ns^{a_n}]=[\alpha_0:\al_1
s^{a_1-a_0}:\cdots:\alpha_ns^{a_n-a_0}]$).

\begin{proposition}\label{affine}
La vari\'et\'e $X_{\A,1}$ d\'epend uniquement de la g\'eom\'etrie affine de l'ensemble $\A$. En d'autres termes, soit $\A\subset \Z^d$, $\mathcal{B}\subset \Z^e$ et
$T:\Z^d\mapsto \Z^e$ une application affine injective et telle que $T(\A)=\mathcal{B}$. Alors $X_{\A,1}$ s'identifie naturellement \`a $X_{\mathcal{B},1}$.
\end{proposition}
\begin{proof}
 Voir \cite[proposition 1.2, p. 167]{GKZ}.
\end{proof}

\begin{example}
 Dans $\p^2(\overline{\Q})$, on consid\`ere la sous-vari\'et\'e d\'efinie par le polyn\^ome homog\`ene suivant:
\[
 P(x,y,z)=y^2-zx
\]
C'est une vari\'et\'e torique au sens de Gelfand, Kapranov et Zelevinsky. En effet c'est l'adh\'erence de Zariski de l'image du morphisme suivant:
\begin{align*}
\mathbb{T}^1&\longrightarrow \mathbb{P}^2(\overline{\Q})\\
 s&\lra
[1:s:s^2].
\end{align*}

\end{example}

Soient $n\in \N_{\geq 2}$ et $d\in \{1,\ldots,n-1\}$. 
On pose $A$, la matrice d'ordre $n\times d $ suivante:
 \[
  A=\left(
     \raisebox{0.5\depth}{%
       \xymatrixcolsep{1ex}%
       \xymatrixrowsep{1ex}%
       \xymatrix{
         a_1 \\
         a_2 \\
\vdots\\
a_n
       }%
     }
   \right).
 \]
D'apr\`es \eqref{affine}, on peut supposer qu'elle est de rang $d$.
On pose $\A:=\{a_1,\ldots,a_n\}$. C'est une sous-famille de vecteurs de $\Z^d$,
qu'on fixe dans la suite.  On suppose que $L_{\A}=\Z^d$.

\begin{lemma}\label{irreductible}
 Soit $\nu=(\nu_1,\ldots,\nu_n)\in \bigl(\Z\setminus\{0\}\bigr)^{n} $, alors
\[
T:=x^\nu-1,
\]
est irr\'eductible dans $\Z\bigl[x_1,x_1^{-1},\ldots,x_n,x^{-1}_n\bigr]$ si et seulement si $\nu_1,\ldots,\nu_n$ sont premiers entre eux.
\end{lemma}
\begin{proof}
Soit $\nu_1,\ldots,\nu_n$ $n$ entiers non tous nuls. On consid\`ere le polyn\^ome de Laurent $T$
suivant:
\begin{equation}
 T\bigl(x_1,\ldots,x_n\bigr)=x_1^{\nu_1}\cdots x_n^{\nu_n}-1.
\end{equation}

Si l'on note par $d$ le plus grand diviseur commun des $\nu_1,\ldots,\nu_n$,  alors $T$ est
r\'eductible,  si $d> 1$. En effet, on a $T= \Bigl(x_1^{\frac{\nu_1}{d}}\cdots
x_n^{\frac{\nu_n}{d}}\Bigr)^d-1$ qui est  r\'eductible dans
$\Z\bigl[x_1,x_1^{-1},\ldots,x_n,x_n^{-1}\bigr]$.\\

S'il existe $p_1,\ldots,p_n$ $n$ entiers tels que $p_1\nu_1+\ldots+p_n\nu_n=1$, on va montrer que
$T$ est irr\'eductible. Supposons  qu'il existe $P$ et $Q$ deux polyn\^omes de Laurent dans
$\Z\bigl[x_1,x_1^{-1},\ldots,x_n,x_n^{-1}\bigr]$ tels que:
\begin{equation}\label{irreductible1}
 T\bigl(x_1,\ldots,x_n\bigr)=P\bigl(x_1,\ldots,x_n\bigr) Q\bigl(x_1,\ldots,x_n\bigr).
\end{equation}

On consid\`ere l'ensemble suivant $E:=\bigl\{1\leq i\leq n\,\bigl|\; p_i\neq 0\bigr\}$ et soit $l$ son cardinal. Quitte \`a r\'eordonner les indices, on peut supposer que $E=\{1,2,\ldots,l \}$. On pose:
\begin{equation*}
 \left\{
\begin{array}{rl}
x_i=y_i^p & \text{si } 1\leq i\leq l\\
x_i=y_i & \text{si } l+1\leq i\leq n.\\
\end{array} \right.
\end{equation*}
En rempla\c{c}ant dans   \eqref{irreductible1}, on a  dans $\Z\bigl[y_1,y_1^{-1},\ldots,y_n,y_n^{-1}\bigr]$:
\[\begin{split}
y_1^{p_1\nu_1}\cdots y_l^{p_l\nu_l} y_{l+1}^{\nu_{l+1}}\cdots y_{n}^{\nu_n}-1=
P\bigl(y_1^{p_1},\ldots,y_l^{p_l},y_{l+1},\ldots,y_n \bigr)Q\bigl(y_1^{p_1},\ldots,y_l^{p_l},y_{l+1},\ldots,y_n \bigr).
\end{split}
\]

Comme $ p_1\nu_1+\ldots+p_n\nu_n=1$, l'\'egalit\'e ci-dessus devient:
\begin{equation}\label{eq}
\begin{split}
y_1 \Bigl({\frac{y_2}{y_1}}\Bigr)^{{}^{p_2\nu_2}}\cdots \Bigl({\frac{y_2}{y_1}}\Bigr)^{{}^{p_l\nu_l}}y_l^{{}^{p_l\nu_l}} y_{l+1}^{{}^{\nu_{l+1}}}\cdots y_{n}^{{}^{\nu_n}}-1=
P\bigl(y_1^{p_1},\ldots,y_l^{p_l},y_{l+1},\ldots,y_n \bigr) Q\bigl(y_1^{p_1},\ldots,y_l^{p_l},y_{l+1},\ldots,y_n \bigr).
\end{split}
\end{equation}

En posant dans \eqref{eq},
\begin{equation*}
 \left\{
\begin{array}{rl}
z_1=y_1 & \\
z_i=\frac{y_i}{y_1} & \text{si }\; 2\leq i\leq l\\
z_i=y_i & \text{si }\; i \geq l+1,
\end{array} \right.
\end{equation*}
on obtient, dans $\Z\bigl[z_1,z_1^{-1},\ldots,z_n,z_n^{-1}\bigr]$,:
\begin{equation}\label{eq11}
 z_1 z_2^{p_2\nu_2}\cdots z_l^{p_l\nu_l} z_{l+1}^{\nu_{l+1}}\cdots z_n^{\nu_n}-1=\widetilde{P}
\widetilde{Q},
\end{equation}
o\`uon a pos\'e $\widetilde{P}=P\bigl(z_1,(z_1z_2)^{p_2}\ldots,(z_1z_l)^{p_l},z_{l+1},\ldots,z_n
\bigr)$ et $\widetilde{Q}=Q\bigl(z_1,(z_1z_2)^{p_2}\ldots,(z_1z_l)^{p_l}$ $,z_{l+1},\ldots,z_n
\bigr)$.\\

Quitte \`a remplacer $z_i$ par $z_i^{-1}$ pour $1\leq i\leq n$, on peut supposer que  $z_1
z_2^{p_2\nu_2}\cdots z_l^{p_l\nu_l} z_{l+1}^{\nu_{l+1}}\cdots z_n^{\nu_n}-1 $ est un polyn\^ome de
$\Z\bigl[z_1,\ldots,z_n\bigr]$. Il existe $\widetilde{\nu}\in\N^n$ et $\widetilde{\mu}\in \N^n$
tels que
\[
\widetilde{P}=\frac{P_1}{z^{\widetilde{\nu}}}\quad\text{et}\quad \widetilde{Q}=\frac{Q_1}{z^{\widetilde{\mu}}},
\]
avec $P_1=\sum_{\nu\in \N^n}a_\nu z^\nu $ et $Q_1=\sum_{\nu\in \N^n}b_\nu z^\nu$ sont deux \'el\'ements de $\Z\bigl[z_1,z_2,\ldots,z_n \bigr]$ et tels que
\[
z_1 z_2^{p_2\nu_2}\cdots z_l^{p_l\nu_l} z_{l+1}^{\nu_{l+1}}\cdots z_n^{\nu_n}-1 =\frac{P_1}{z^{\widetilde{\nu}}}\frac{Q_1}{z^{\widetilde{\mu}}}.
\]
Si l'on note par $S(P)$ (resp. $S(Q)$) l'ensemble des $\nu\in \N^n$ tels que $a_\nu\neq 0$ (resp.
$b_{\nu}\neq 0$) alors, on voit que:
\begin{equation}\label{gtrtar}
\nu+\nu'\geq \widetilde{\nu}+\widetilde{\mu}\quad \forall\,\, \nu\in S(P),\,\forall\,\, \nu'\in
S(Q).
\end{equation}

On peut supposer que
\[
\widetilde{\nu}_i+\widetilde{\mu}_i\neq 0\quad\,\forall\, \, 1\leq i\leq n.
\]
Donc, il existe $i_0$ et $\nu^{(0)}\in S(P)$ et $\nu'^{(0)}\in S(Q)$ tels que
\[
\inf_{\substack{0\leq i\leq n \\
\nu\in S(P),\, \nu'\in S(Q)}}\bigl(\nu_i+\nu'_i \bigr)=\nu^{(0)}_{i_0}+\nu'^{(0)}_{i_0}\neq 0.
\]
Supposons, par exemple que $\nu^{(0)}_{i_0}\geq \widetilde{\nu}_{i_0}$ (ce qui est possible d'apr\`es
 \eqref{gtrtar}). On \'ecrit alors
\[
\frac{P_1}{z^{\widetilde{\nu}}}\frac{Q_1}{z^{\widetilde{\mu}}}=\biggl(\frac{P_1}{z_{i_0}^{\widetilde{\nu}_{i_0}+\widetilde{\mu}_{i_0}-\min(\widetilde{\mu}_{i_0},\nu_{i_0}^{'(0)})}\widehat{z}^{\widetilde{\nu}}}\biggr)\biggl(\frac{Q_1}{z^{\min(\widetilde{\mu}_{i_0},\nu_{i_0}^{'(0)})}_{i_0}\widehat{z}^{\widetilde{\mu}}}\biggr)
\]
o\`u$\widehat{z}^{\widetilde{\nu}}$ et $\widehat{z}^{\widetilde{\mu}}$ sont des mon\^omes qui ne contiennent pas une puissance de $z_{i_0}$. Comme
\[
\nu_{i_0}^{(0)}+\nu_{i_0}'^{(0)}\geq \widetilde{\nu}_{i_0}+\widetilde{\mu}_{i_0},
\]
et qu'on a suppos\'e $\nu^{(0)}_{i_0}\geq \widetilde{\nu}_{i_0}$
alors
\[
\frac{P_1}{z_{i_0}^{\widetilde{\nu}_{i_0}+\widetilde{\mu}_{i_0}-\min(\widetilde{\mu}_{i_0},\nu_{i_0}^{'(0)})}\widehat{z}^{\widetilde{\nu}}}\quad \text{et}\quad \frac{Q_1}{z^{\min(\widetilde{\mu}_{i_0},\nu_{i_0}^{'(0)})}_{i_0}\widehat{z}^{\widetilde{\mu}}},
\]
sont deux \'el\'ements de
\[
\Z\bigl[z_1^\pm,\ldots,z_{i_0-1}^\pm,z_{i_0},z_{i_0+1}^\pm,\ldots,z_n^\pm  \bigr].
\]
On conclut par r\'ecurrence qu'on peut trouver $\widetilde{\nu}',\widetilde{\mu}'\in \N^n$ tels que $\widetilde{\nu}+\widetilde{\mu}=\widetilde{\nu}'+\widetilde{\mu}'$ et
\[
z_1 z_2^{p_2\nu_2}\cdots z_l^{p_l\nu_l} z_{l+1}^{\nu_{l+1}}\cdots z_n^{\nu_n}-1 =\frac{P_1}{z^{\widetilde{\nu}'}}\frac{Q_1}{z^{\widetilde{\mu}'}},
\]
avec
\[
\widetilde{P}_1:=\frac{P_1}{z^{\widetilde{\nu}'}}\in \Z\bigl[z_1,\ldots,z_n\bigr],\quad \widetilde{Q}_1:=\frac{Q_1}{z^{\widetilde{\mu}'}}\in \Z\bigl[z_1,\ldots,z_n\bigr].
\]
  Si $\deg_{z_1}\widetilde{P}_1\geq 1$,  on \'ecrit $\widetilde{P}_1=R_1\cdot z_1-R_2$, avec
$R_1\neq 0$ et $R_2\in \Z[z_2,\ldots,z_n]$. De \eqref{eq11}, on  d\'eduit que $R_2\cdot
\widetilde{Q}_1=1$, donc $\widetilde{Q}_1$ est une constante.  Vu les changements de variables
qu'on a effectu\'e,  il existe  des entiers $\mu_1,\ldots,\mu_n$ et une constante $c$ tels que
$Q=c\cdot x_1^{\mu_1}\cdots x_n^{\mu_n}$. On conclut que $T$ est irr\'eductible dans
$\Z\bigl[x_1,x_1^{-1},\ldots,x_n,x_n^{-1}\bigr]$.
\end{proof}
\begin{lemma}\label{variete} Si $A$ est de rang $d$, alors il existe $n-d $ vecteurs de $\Z^n$,  $w_1,\ldots,w_{n-d}$, tels que
\begin{equation}\label{eqq}
 X_{\A,1}^\circ=\bigl\{x\in \T^n \,\bigl|\, x^{w_{+i}}=x^{w_{-i}}, \forall\,\, 1\leq i\leq n-d
 \,\bigr\},
\end{equation}

o\`u $w_{+i}:=\bigl(\max(0,w_{i1}),\ldots,\max(0, w_{in})\bigr)$ et $w_{- 
i}:=\bigl(\max(0,-
w_{i1}),\ldots,
\max(0,- w_{in})\bigr)$ pour $i=1,\ldots,n-d$ $($de sorte qu'on a  
$w_i=w_{+i}-w_{-i} $  $)$.  R\'eciproquement, si l'on se donne un 
ensemble de $n-d$ vecteurs $w_1,\ldots,w_{n-d}$ de $\Z^n$ qu'on peut
compl\'eter en une base $\{w_1,\ldots,w_n\}$ de $\Z^n$, alors il 
existe $\A$ un ensemble de $n$ vecteurs de $\Z^d$ tel qu'on a  
\eqref{eqq}.

\end{lemma}

\begin{proof}

Soit ${}^t A  $ la matrice transpos\'ee de $A$, on a $\ker({}^t A) $ est un sous-groupe satur\'e de $\Z^n$ de rang $n-d$, par le th\'eor\`eme de structure des modules de type fini sur un anneau principal, voir par exemple \cite[th\'eor\`eme
7.8]{Lang}, il existe une base $\bigl\{w_1,\ldots,w_n\bigr\}$  de $\Z^n$ telle que
$\bigl\{w_1,\ldots,w_{n-d}\bigr\}$ soit une base de $\ker({}^t A)$ et donc si l'on pose:
\[
 P_i(x)=x^{w_i}-1 \in \Z\bigl[x_1,x_1^{-1},\ldots,x_n,x_n^{-1}\bigr]\quad \forall\,\,
i=1,\ldots,n-d,
\]
alors par  le lemme \eqref{irreductible}, ces polyn\^omes sont irr\'eductibles dans $\Z\bigl[x_1,x_1^{-1},\ldots,x_n,x_n^{-1}\bigr] $.\\

Montrons que
\[
 P_j(X_{\A,1}^\circ)=\{0\}, \;\text{avec}\; j=1,\ldots,n-d.
\]

Soit donc, $x\in X_{\A,1}^\circ$. Par d\'efinition, il existe $t\in \T^d$ tel que
$x=(1,t^{a_1},\ldots,t^{a_n})$. On a $x^{w_j}=\prod_{i=1}^n t^{a_i
w_{ji}}=\prod_{i=1}^n\prod_{k=1}^d t^{a_{ik}w_{ji}}=\prod_{k=1}^d t^{<a_k,w_j>}=1$, pour tout
$j=1,\ldots,n-d$.\\

R\'eciproquement, soit $x\in \T^n $ tel que
\begin{equation}\label{zeroPjjj}
 x^{w_i}-1=0,\quad \forall\,  i=1,\ldots, n-d.
\end{equation}

Si $i=1,\ldots,n$, on choisit un r\'eel qu'on note $\arg(x_i)$, tel que
$x_i=|x_i|\exp\bigl(2\pi\sqrt{-1} \arg(x_i)\bigr)$ et on pose
$\arg(x)=\bigl(\arg(x_1),\ldots,\arg(x_n)\bigr)$ et
$\log|x|=\bigl(\log|x_1|,\ldots,\log|x_n|\bigr)$. Alors,
le syst\`eme d'\'equations \eqref{zeroPjjj} est \'equivalent aux
  deux syst\`emes d'\'equations suivants:
\begin{equation}\label{eq1}
 <w_i,\log|x|>=0 ,\quad i=1,\ldots,n-d,
\end{equation}
\begin{equation}\label{eq2}
 k_i:=<w_i,\arg(x)>\in \Z\quad i=1,\ldots,n-d.
\end{equation}
On v\'erifie que $\ker({}^t A )_\R$ et $A\cdot \R^d$ sont orthogonaux pour le produit scalaire
standard de $\R^n$, et comme $A$ est de rang $d$ et que ses vecteurs lignes  engendrent $\Z^d$
par hypoth\`ese, alors on a la somme directe orthogonale suivante:
\begin{equation}\label{somme}
 \R^n=\ker({}^t A )_\R\oplus^\perp A\cdot \R^d,
\end{equation}
On en  d\'eduit que \eqref{eq1} admet une solution, c'est \`a dire, il existe $y\in \R^d$ tel que
$|x|=\exp(A\cdot y) $.\\

Montrons qu'il existe $n$ entiers $v_1,\ldots v_n$ tels que 
\begin{equation}\label{x13}
\sum_{j=1}^nv_j<w_i,w_j>=k_i \quad
i=1,\ldots,n-d,
\end{equation}
Pour cela, on pose $k:=(k_1,\ldots,k_{n-d},0,\ldots,0)\in \Z^n$ et $G:=(<w_i,w_j>)_{1\leq i,j\leq
n}$. On a, la matrice $G$ est inversible et son d\'eterminant vaut 1, en fait, on peut \'ecrire
$G={}^tW\cdot W$ avec $W$ est une matrice $n\times n$ inversible dans l'espace des matrices \`a
coefficients dans $\Z$ et dont les colonnes sont les vecteurs $w_1,\ldots,w_n$, qui forment une
base de $\Z^n$ par hypoth\`ese. Par cons\'equent,
on peut
trouver $\overline{v}\in\Z^n$ tel que \[G\cdot \overline{v}=k.\]

On consid\`ere $\arg(x)-\overline{v}$ dans $\R^n$. Il existe $\theta\in \R^d$ et $q\in \ker({}^t
A)_\R$ tels que
\[
 \arg(x)-\overline{v}=q+A\theta.
\]
Mais comme $<\arg(x),w_i>=k_i=<\overline{v},w_i>$ , alors $<q,w_i>=0$ et cela pour  tout
$i=1,\ldots,n-d$
. C'est \`a dire que $q$ est orthogonal \`a $\ker({}^t A)_{\R}$, donc $q=0$.\\

Si l'on pose
\[
 z:=y+2\pi \sqrt{-1}\theta,
\]
alors
\[\begin{aligned}
 \exp(Az)&=\exp(Ay+2\pi \sqrt{-1}
A\theta)\\
&=|x|\exp\bigl(2\pi \sqrt{-1}(\arg(x)-\overline{v})\bigr)\\
&=|x|\exp\bigl(2\pi \sqrt{-1}(\arg(x)\bigr)\quad \text{car}\; \overline{v}\in \Z^n\\
&=x.
\end{aligned}
\]
Donc en posant $t=e^z$, alors
\[
 x=\bigl(e^{<a_1,z>},\ldots,e^{<a_n,z>} \bigr)=\bigl(t^{a_1},\ldots,t^{a_n} \bigr)
\]
c'est \`a dire que
\[
 x\in X_{\A,1}^\circ.
\]
\end{proof}

\section{Hauteurs canoniques des sous-vari\'et\'es toriques}\label{x200}

Soit $\beta=(\beta_0,\ldots,\beta_n)
\in (\overline{\Q}^\ast)^{n+1}$. On peut supposer
que $\beta_0=1$\footnote{ On v\'erifie imm\'ediatement que  $X_{\A,\beta}=X_{\A,\gamma\cdot 
\beta}$ o\`u $\gamma\cdot \beta:=(\gamma\cdot \beta_0,\ldots,
\gamma\cdot \beta_n)$ pour tout $\gamma\in \overline{\Q}^\ast$}. 
Rappelons qu'on a suppos\'e que $L_{\A}\simeq \Z^d$, et donc
$X_{\A,\beta}$ est int\`egre. On pose
$K:=\Q(\beta_1,\ldots,\beta_n)$
et soit $\mathcal{O}_K$ l'anneau
des entiers de $K$. Soient $w_1,\ldots,w_{n-d},$
$n-d$ vecteurs de 
$\Z^d$ comme dans Lemme \eqref{variete}.  

 Pour tout $i\in 
\{1,2,\ldots,n-d\}$, on note par $Q_i$ (resp. $R_i$)
le polyn\^ome dans 
$K[x_1,x_2,\ldots,x_n]$ (resp. $K[T_0,\ldots,T_n]$) suivant
\begin{equation}\label{x16}
\begin{split}
 Q_i&:=\beta^{-w_{+i}}x^{w_{+i}}-\beta^{-w_{-i}} x^{w_{-i}},\quad R_i(T_0,T_1,\ldots,T_n):=\beta^{-w_{+i}}T^{u_{+i}}-\beta^{-w_{-i}}
T^{u_{-i}}\\
\text{avec}\quad u_i&:=(-\sum_{k=1}^nw_{ik},w_{i1},\ldots,w_{in}),
\end{split}
\end{equation}
o\`u $\beta^{w_\ast}:=\prod_{k=1}^n \beta_k^{w_{\ast,k}}$ 
et $T^{u_\ast}:=\prod_{k=0}^n T_k^{u_{\ast,k}}$.

On note par $\X_{\A,\beta}$ le sous-sch\'ema ferm\'e int\`egre
 de
$\p^n_{\mathcal{O}_K}=\mathrm{Proj}\bigl(\mathcal{O}_K\bigl[T_0,T_1,
\ldots,T_n\bigr]\bigr)$ d\'efini par l'id\'eal
homog\`ene, $\mathcal{O}_K[T_0,\ldots,T_n]\cap 
I$ o\`u  $I$ est l'id\'eal de $K[T_0,\ldots,T_n]$  
d\'efinissant $X_{\A,\beta}$. 
\begin{definition}\label{condA}
Soit $X_{\A,\beta}$ une vari\'et\'e torique projective int\`egre 
de dimension
$d$ dans $\p^n(\overline{\Q})$ et $K$ un corps de d\'efinition de
$X_{\A,\beta}$. On dit que $X_{\A,\beta}$ satisfait
la condition $\mathscr{A}$, si les  polyn\^omes homog\`enes $R_1,
\ldots,R_{n-d}$
 de $K[T_0,\ldots,T_n]$ d\'efinisssent
$X_{\A,\beta}$  et 
 pour tout $k\geq 2$, la vari\'et\'e d\'efinie par 
$R_k,\ldots,R_{n-d}$ est int\`egre.
\end{definition}

\begin{remarque}\label{nedepenbeta}
La condition $\mathscr{A}$ ne d\'epend pas de $\beta=(1,\beta_1,
\ldots,\beta_n)$. En effet,
si l'on consid\`ere $\theta$ le automorphisme de $\p^n_K$
 d\'efini par  
 l'isomorphisme d'alg\`ebres suivant 
 \[
K[T_0,\ldots,T_n]\rightarrow K[T_0,\ldots,T_n],\quad
T_i\mapsto \beta_iT_i, 
 \]
 alors on v\'erifie que 
 $\theta (X_{\A,\beta})=X_{\A,1}$. En d'autres termes, on peut
 supposer que $\beta=(1,\ldots,1)$ dans la d\'efinition.
 \end{remarque}
\begin{example}\label{exemple2}
Gr\^ace \`a la notion de matrice mixte et dominante, on peut
construire des exemples de vari\'et\'es toriques $X_{\A,\beta}$
v\'erifiant D\'efinition \eqref{condA} (voir 
la section \ref{x100} pour plus de d\'etails.)
\end{example}

Soient $R_1,\ldots,R_{n-d}$ comme avant.
 On suppose qu'il existe un corps de nombres $F$ tel que 
son anneau des entiers est factoriel et que $\beta\in (F^\ast)^{n+1}$. Sous
cette hypoth\`ese, on a le lemme suivant 

\begin{lemma}\label{x5}
Pour tout $i=
1,\ldots,n-d$, il existe $\tau_i\in K$ tel que
 \[
 (R_i)\cap \mathcal{O}_K[T_0,\ldots,T_n]=(\tau_i R_i).
 \]
\end{lemma}

\begin{proof}
Soit $i\in \{1,\ldots,n-d\}$. On a 
$R_i=\beta^{-w_{+i}}T^{u_{+i}}-\beta^{-w_{-i}}T^{u_{-i}}=
\beta^{-w_{+i}}(T^{u_{+i}}-\beta^{w_i}T^{u_{-i}})$. Comme
$\beta\in (F^\ast)^{n+1}$ avec $\mathcal{O}_F$ est factoriel, alors 
il existe $a_i$ et $b_i$ dans $\mathcal{O}_F$ premiers entre eux tels que
$\beta^{w_i}=\frac{a_i}{b_i}$. On pose
$\tau_i:=\beta^{w_{+i}}b_i$. Montrons que
\[
 (R_i)\cap \mathcal{O}_K[T_0,\ldots,T_n]=(\tau_i R_i).
\]
Cela peut se d\'eduire du point $(3.)$ de l'exemple \eqref{x9}.
\end{proof}

Ci-dessous, on construit des exemples de corps de nombres $F$ et 
des \'el\'ements  $\beta\in (F^\ast)^n$ qui v\'erifient le r\'esultat
du Lemme \eqref{x5}. En d'autres termes, on ne suppose pas n\'ecessairement
que l'anneau des entiers de $F$ soit factoriel.
\begin{example}\label{x9}

\begin{enumerate}
\item 
Soit $L_\A\subset \Z^d$ le sous-module engendr\'e par les vecteurs $a_1,a_2,\ldots,a_n$, que l'on  suppose isomorphe \`a $\Z^d$. 
On note par $H:=\{x\in \R^n|\, \langle x,w_i\rangle =0,\, \text{pour}\, i=1,\ldots,n-d  \}$\footnote{Ici $\langle \cdot,\cdot \rangle$ est le produit scalaire
usuel sur $\R^n$.}.  
Soit $\beta=(1,\beta_1,\ldots,\beta_n)\in (\overline{\Q}^\ast)^{n+1}$ avec  
$\beta_k:=e^{v_k} \mu_k$ pour $k=1,\ldots,n$ avec 
$v:=(v_1,\ldots,v_n)\in H\cap \log ((\overline{\Q}^\ast_{>0})^n)$ et $\mu_i$ est une racine
de l'unit\'e pour $i=1,\ldots,n-d$. On a 
$\beta^{w_k}=e^{\langle v,w_k\rangle } \prod_{i=1}^n \mu_i^{w_{ki}}= 
\prod_{i=1}^n \mu_i^{w_{ki}}$, qui est une racine de l'unit\'e et donc
un \'el\'ement de $\mathcal{O}_K$.  Donc, si l'on pose  
$\tau_k:= \beta^{w_{+k}}$ pour tout $k=1,\ldots,n-d$, alors 
$\tau_k R_k=T^{w_{+k}}-\beta^{w_k}T^{u_{-k}} \in \mathcal{O}_K[T_0,
\ldots,T_n]$ et  on v\'erifie que
 \[
 (R_k)\cap \mathcal{O}_K[T_0,\ldots,T_n]=( \tau_k R_k)\quad
 \forall k=1,\ldots,n-d,
 \]
 avec $K=\Q(\beta_1,\ldots,\beta_n)$.
 \item Si $\beta=(1,\beta_1,\ldots,\beta_n)\in (\overline{\Q}^\ast)^{n+1}$ 
avec $\beta^{-w_{+i}},\beta^{-w_{-i}}\in \Q$  pour tout $i=1,\ldots,n-d$. Alors
 $\beta$ v\'erifie le r\'esultat du Lemme \eqref{x5}. En effet, cela
 r\'esulte du lemme suivant
 \begin{lemma}
 Soit $v=(v_1,\ldots,v_n)\in \Z^n$ avec 
 $\gcd(v_1,\ldots,v_n)=1$. Soient $a,b\in \Z$ avec $\gcd(a,b)=1$, alors  
 l'id\'eal $(bx^v+a)$ est premier dans $\Z[x_1^{\pm},\ldots,
 x_n^{\pm}]$.
 \end{lemma}
 \begin{proof}
Si $P\in \Z[x_1^{\pm},\ldots,
 x_n^{\pm}]$ on appelle le contenu de $P$ et on le note par
 $c(P)$, l'\'el\'ement de $\Z$ d\'efini comme \'etant le 
 plus grand commun diviseur des coefficients de $P$. On v\'erifie
 que $c(a P_1)=ac(P_1)$ et $c(P_1P_2)=c(P_1)c(P_2)$ pour tous 
 $a\in \Z$ et $P_1,P_2\in \Z[x_1^{\pm},\ldots,
 x_n^{\pm}]$.
 
  Montrons que 
 $(bx^v+a)$ est premier dans $ \Z[x_1^{\pm},\ldots,
 x_n^{\pm}]$. Soient $P,Q\in  \Z[x_1^{\pm},\ldots,
 x_n^{\pm}]$ tels que $PQ\in (bx^u+a)$. Notons que
 $(bx^v+a)$ est premier  $\Q[x_1^{\pm},\ldots,
 x_n^{\pm}]$, cela r\'esulte du Lemme \eqref{irreductible}\footnote{On peut
 v\'erifier que le lemme \eqref{irreductible} est valable si l'on
  consid\`ere $\overline{\Q}$ \`a la place de $\Z$.}.
  Donc, on peut supposer qu'il existe 
 $q\in \Z\setminus\{0\}$ et $R\in  \Z[x_1^{\pm},\ldots,
 x_n^{\pm}]$ tels que
 $q P=(bx^v+a)R$. Par suite,
 $q c(P)=c(R)$ (puisque $\gcd(a,b)=1$). Par cons\'equent
 $\frac{R}{q}\in \Z[x_1^{\pm},\ldots,
 x_n^{\pm}]$ et donc $P\in (bx^v+a)$.
 \end{proof}

 \item L'exemple suivant g\'en\'eralise les deux exemples 
 pr\'ec\'edents:  Soit $v=(v_1,\ldots,
 v_n)\in \Z^n$ avec $\gcd(v_1,\ldots,v_n)=1$. Soit $F$ un corps de nombres et soit $\mathcal{O}_F$
 son anneau des entiers. 
 Soient $a,b\in \mathcal{O}_F\setminus\{0\}$
  avec $(a)+(b)=\mathcal{O}_F$\footnote{
 Par exemple, si $a\in \mathcal{O}_F\setminus\{0\}$ alors
 tout $b\in \Z\setminus\{0\}$ avec $\gcd(b, N_F(a))=1$, (o\`u 
 $N_F(a)$ est 
 la norme de 
 $a$ dans $F$) on a $(a)+(b)=\mathcal{O}_F$.}.
 Montrons que 
 l'id\'eal $(bx^v+a)$ est premier dans $\mathcal{O}_F[x^{\pm}_1,
 \ldots,x_n^{\pm}]$.  Soient $P,Q\in  \mathcal{O}_F[x_1^{\pm},\ldots,
 x_n^{\pm}]$ tels que $PQ\in (bx^v+a)$. Notons que
 $(bx^v+a)$ est premier  $F[x_1^{\pm},\ldots,
 x_n^{\pm}]$, cela r\'esulte de \eqref{irreductible}.
  Donc, on peut supposer qu'il existe 
 $r\in \mathcal{O}_F\setminus\{0\}$ et $R\in  \mathcal{O}_F
 [x_1^{\pm},\ldots,
 x_n^{\pm}]$ tels que
 $r P=(bx^v+a)R$. Si $S\in \mathcal{O}_F[x_1^{\pm},\ldots,x_n^{\pm}]$
 alors on note par $\widetilde{c}(S)$ l'id\'eal dans $\mathcal{O}_F$
 engendr\'e par les coefficients de $S$. Par suite, $r P=(bx^v+a)R$ 
 implique que
 $r c(P)=c(R)$ (puisque $(a)+(b)=\mathcal{O}_F$). Par cons\'equent
 $\frac{R}{r}\in \mathcal{O}_F[x_1^{\pm},\ldots,
 x_n^{\pm}]$ et donc $P\in (bx^v+a)$.\\

 Pour tout $j=1,\ldots,n-d$, soient
 $a_j$ et $b_j$ deux \'el\'ements de $\mathcal{O}_F$ avec
 $(a_j)+(b_j)=\mathcal{O}_F$. Il existe 
 $\beta=(1,\beta_1,\ldots,\beta_n)\in (\overline{\Q}^\ast)^{n+1}$ avec
  tel que 
 $\beta^{w_j}=\frac{a_j}{b_j}$  pour tout $j=1,\ldots,n-d.$ L'existence de $\beta$
r\'esulte du fait que l'application suivante
 \[
 (\C^\ast)^{n}\longrightarrow (\C^\ast)^{n},\quad
 \gamma\mapsto (\gamma^{w_i})_{1\leq i\leq n},
 \]
 
 est un isomorphisme (Rappelons que 
 $\{w_1,\ldots,w_n \}$ est une base de $\Z^n$). Donc, si
 l'on pose $\tau_{i}:=b_i\beta^{w_{+i}}$  alors 
 $(\tau_{i} R_i)=(b_iT^{u_{+i}}-a_iT^{u_{-i}})$ qui est premier, et 
 on v\'erifie que
 \[(R_i)\cap \mathcal{O}_K[T_0,\ldots,T_n]=(\tau_{i} R_i)\quad
 \forall i=1,\ldots,n-d.\]
 
\end{enumerate}
\end{example}
\begin{remarque}
\rm{Soit $\beta\in (F^\ast)^{n+1}$ avec $F$ un corps de nombres tel que
son anneau des entiers soit factoriel, alors on v\'erifie 
 que pour tout choix de $\A$ et en gardant les notations du \ref{x16}, il
 existe $\tau_i \in K$ pour tout $i=1,\ldots,n-d$ tels que
 \[
 (R_i)\cap \mathcal{O}_K[T_0,\ldots,T_n]=(\tau_i R_i)\quad 
 i=1,\ldots,n-d.
 \] 
 La preuve de ce fait est une simple adapation de celle du Lemme 
 \eqref{x5}. }
\end{remarque}

Dans la suite, on aura besoin du lemme d'alg\`ebre suivant 
\begin{lemma}\label{lemmemodele}
Soit $A$ un anneau commutatif unitaire int\`egre, et $F$ son
anneau de fractions. On note par $A[T_0,\ldots,T_n]$
(resp. $F[T_0,\ldots,T_n]$) l'anneau des polyn\^omes 
\`a $n+1$ variables \`a coefficients dans $A$ (resp. $F$). Soit
$l\in \{1,\ldots,n\}$.  
Soient $R_1,\ldots,R_l$, $l$ polyn\^omes homog\`enes dans $F[T_0,\ldots,T_n]$
tels que
\begin{enumerate}
\item  La suite d'id\'eaux suivante $\bigl((R_k,R_{k+1},\ldots,R_l)\bigr)_{k=1,\ldots,l}$  d\'efinit une suite de sous-vari\'et\'es de $\p^n_F$,
 strictement
d\'ecroissante pour l'inclusion.
 \item $( R_k)$ est premier dans $F[T_0,\ldots,T_n]$ pour tout
 $k=1,\ldots,l$.
\item  Il
existe $\si_1,\ldots,\si_l\in F$ tels que
$(R_k)\cap A[T_0,\ldots,T_n]=(\si_k R_k)$ pour tout
 $k=1,\ldots,l$.
\end{enumerate}  
Alors,
$(R_k,R_{k+1},\ldots,R_l)\cap A[T_0,\ldots,T_n]=(\si_kR_k,\ldots,\si_lR_l)$ et 
$(\si_kR_k,\ldots,\si_lR_l)$ est premier dans 
$A[T_0,\ldots,T_n]$ pour tout $k=1,\ldots,n$.

\end{lemma}

\begin{proof}
Posons $J_k:=(\si_k R_k,\ldots,
\si_l R_l)$ pour tout $k=1,\ldots,l$.
Montrons ce r\'esultat par r\'ecurrence sur $k$. Si 
$k=l$, alors le r\'esultat est vrai par hypoth\`ese. Supposons
que le lemme est vrai pour un certain $k+1$. Le morphisme
suivant \[
(\si_k R_k)\rightarrow \frac{J_{k}}{J_{k+1}},\quad
\]
qui envoie un \'el\'ement de $(\si_k R_k)$ vers sa classe
dans $\frac{J_{k}}{J_{k+1}}$ est clairement 
surjectif. Soit $Q\in A[T_0,\ldots,T_n]$ tel que 
$Q \cdot \si_kR_k\in J_{k+1}$. Comme $J_{k+1}$ est premier (par
hypoth\`ese de r\'ecurrence) alors on a n\'ecessairement 
$Q\in J_{k+1}$, car sinon $\si_kR_k\in J_{k+1}=(\si_{k+1} R_{k+1},\ldots,
\si_lR_l)$ et par cons\'equent
 $(R_k,R_{k+1},\ldots,R_l)=(R_{k+1},R_{k+2},\ldots,R_l)$ dans
 $F[T_0,\ldots,T_n]$, ce qui 
contredit l'hypoth\`ese $(1.)$. Donc,
\[
\frac{(\si_kR_k)}{J_{k+1}}\simeq \frac{J_k}{J_{k+1}},
\]
Par $(2.)$ et $(3.)$, on d\'eduit que $\frac{J_k}{J_{k+1}}$  est premier. 
Puisque $\frac{A[T_0,\ldots,T_n]}{J_k}\simeq \frac{A[T_0,\ldots,
T_n]}{J_{k+1}}/\frac{J_k}{J_{k+1}}$. Alors $J_k$ est premier dans
$A[T_0,\ldots,T_n]$. 

On a clairement $(\si_kR_k,\ldots,\si_lR_l)\subset 
(R_k,R_{k+1},\ldots,R_l)\cap A[T_0,\ldots,T_n]
$ pour tout $k=1,\ldots,l$.
 Si $f\in (R_k,R_{k+1},\ldots,R_l)\cap A[T_0,\ldots,T_n]$, 
 alors il existe $L_1,\ldots,L_l\in F[T_0,\ldots,T_n]$ tels que
 $f=\sum_{i=k}^lL_i R_i=\sum_{i=k}^{l}\si_i^{-1}L_i (\si_i R_i)$.
 On peut trouver $\al\in A$ tel que $\al \si_i^{-1}L_i 
 \in A[T_0,\ldots,T_n]$ pour $i=k,\ldots,l$. Par suite, $\al f\in   J_k$. Comme cet id\'eal est premier, alors on a  n\'ecessairement 
 $f\in J_k$. Ce qui termine la preuve du lemme.
\end{proof}

On suppose dans la suite que $X_{\A,\beta}$ satisfait la condition $
\mathscr{A}$ (voir D\'efinition \eqref{condA}) et que $\beta \in (F^\ast)^{n+1}$ avec $F$ est un corps de nombres
dont l'anneau des entiers est factoriel.
  Rappelons que $\Z w_1+\ldots+\Z w_{n-d}$ est un sous-module libre
et satur\'e de $\Z^n$ qu'on a complet\'e en une base  
 $\{w_1,\ldots,w_{n-d},w_{n-d+1},\ldots,w_n\}$  
de $\Z^n$.  On pose $A_1:=A$ et note par 
$b_1,\ldots,b_d$ les lignes de ${}^t A_1$. Pour tout
$k=1,2,\ldots,n-d-1$, soit $\vf_k$ l'homorphisme de $\Z$-modules
suivant

\[
\vf_k:\Z^n\longrightarrow \Z^{n-d-k}\quad 
x\mapsto (<x,w_j>)_{1\leq j\leq n-d-k}\quad \forall 1\leq k\leq n-d-
1.
\]
En notant que $\ker(\vf_k)$ est satur\'e.
 Par le th\'eor\`eme de la base adapt\'ee, on peut trouver $b_{d+1},b_{d+2},
\ldots,b_n\in \Z^n$ tels que $\ker(\vf_k)=\Z b_{k+d}+\cdots
+\Z b_d+\ldots+\Z b_1$, pour $k=1,\ldots,n-d-1$\footnote{Les vecteurs
$b_{d+1},\ldots,b_n$ peuvent \^etre obtenus 
par r\'ecurrence, en utilisant la r\'eduction de Smith. }.  Si l'on note par 
$A_k$ la matrice de taille $(k+d)\times n$ dont la matrice transpos\'ee
${}^t A_k$ est form\'ee par les lignes $b_1,\ldots,b_{k+d}$ pour $k=
1,\ldots,n-d-1$. Alors, 
\[
\ker( {}^tA_k)=\Z w_1+\cdots +\Z w_{n-k}\quad \forall k=1,\ldots,n-d-1.
\] 
Par suite, $X_{A_k,\beta}$ est la vari\'et\'e d\'efinie 
par l'id\'eal $J_k:=(R_k,\ldots,R_{n-d})$ et 
$X_{\A_k,\beta}$ est une hypersurface dans $X_{\A_{k+1},\beta}$ 
donn\'ee
par $R_{k}$ pour tout $ k=1,\ldots,n-d-1$, et 
$X_{\A_{n-d},\beta}$ donn\'ee par $R_{n-d}$ dans $\p^n_K$ par la condition $\mathscr{A}
$.

Pour tout $k=1,\ldots,n-d$,
 on note par $\X_{\A_k,\beta}$ la cl\^oture de Zariski de 
 $X_{\A_k,\beta}$. Autrement dit, $\X_{\A_k,\beta}$ est le 
 sch\'ema projectif d\'efini par l'id\'eal premier homog\`ene
 \[
 J_k\cap \mathcal{O}_k[T_0,\ldots,T_n].
 \]
 Comme on a suppos\'e que $X_{\A,\beta}$ v\'erifie l'hypoth\`ese 
 $\mathscr{A}$ et  que  $\beta\in (F^\ast)^{n+1}$ avec $F$ un corps de nombres
 dont l'anneau des entiers soit factoriel, alors $R_1,\ldots,R_{n-d}$ v\'erifie
les conditions du lemme \eqref{lemmemodele}\footnote{
Les conditions 1. et 3. du lemme \eqref{lemmemodele} sont automatiques dans ce cas. Comme l'application \'evidente 
$(R_i)\rightarrow (R_i,\ldots,R_{n-d})/(R_{i+1},\ldots,R_{n-d})$ est surjective et que $(R_i,\ldots,R_{n-d})$ est premier
car $X_{\A_i,\beta}$ est int\`egre par hypoth\`ese, alors on d\'eduit
la condition $2.$} .  Par suite  $\X_{\A_k,
\beta}$ 
 est le mod\`ele int\`egre pour $X_{\A_k,\beta}$ sur 
 $\mathrm{Spec}(\mathcal{O}_K)$ d\'efini par 
 \[
 (\tau_kR_k,\ldots,\tau_{n-d}R_{n-d}),
 \] 
 
et que  pour tout $k=1,\ldots,n-d-1$,  $\X_{\A_k,\beta}$ 
est une hypersurface dans $\X_{A_{k+1},\beta}$ donn\'ee  par 
le polyn\^ome $s_k:=\tau_{k}R_k\in \mathcal{O}_K\bigl[T_0,\ldots,T_n \bigr]$ et $\X_{\A_{n-d},\beta}$ est l'hypersurface 
dans $\p^n_{\mathcal{O}_K}$ d\'efinie par $s_{n-d}:=\tau_{n-d}
R_{n-d}$.
 On a alors la  
 suite suivante
 de sch\'emas projectives int\`egres sur $\mathrm{Spec}(\mathcal{O}_K)$

 \begin{equation}\label{x3}
 \X_{\A_1, \beta}:=\X_{\A,\beta}
 \subsetneq \X_{\A_2,\beta}\subsetneq \cdots\subsetneq \X_{\A_{n-d},
 \beta}\subsetneq \p^n_{\mathcal{O}_K}.
\end{equation}

 Par la formule due \`a Faltings (voir \cite[(3.2.1) p. 949]{BoGS} pour les
m\'etriques
$\mathcal{C}^\infty$ et \cite{Maillot} pour les m\'etriques non $\mathcal{C}^\infty$) on a pour tout $j=1,\ldots,n-d-1$
\begin{equation}\label{faltings}
\begin{split}
h_{{}_{\overline{\mc}_\infty}}\bigl(\X_{\A_j,\beta}\bigr)&=
\deg(R_{
j})h_{{}_{\overline{\mc}_\infty}}\bigl(\X_{\A_{j+1
},\beta}\bigr)+\sum_{\si:K\rightarrow \C}\int_{\X_{\A_{j+1},\beta}(\C)}\log\bigl\|s_j\bigr\|_{\si,\infty}c_1\bigl(
\overline{\mc}_\infty\bigr)^{j+d},\\
h_{{}_{\overline{\mc}_\infty}}\bigl(\X_{\A_{n-d},\beta}\bigr)&=
\deg(R_{n
-d})h_{{}_{\overline{\mc}_\infty}}\bigl(\p^n_{\mathcal{O}_K}\bigr)+\sum_{\si:K\rightarrow \C}\int_{\p^n(\C)}\log\bigl
\|s_j\|_{\si,\infty} c_1\bigl(
\overline{\mc}_\infty\bigr)^{n}
\end{split}
\end{equation}

Comme la hauteur canonique de $\p^n_{\mathcal{O}_K}$ est nulle
(voir \cite[Proposition 7.1]{Maillot}), alors 
le calcul de la hauteur canonique de $\X_{\A,\beta}$
 r\'esultera de \eqref{faltings}, si l'on trouve une formule explicite pour $
c_1\bigl(
\overline{\mc}_\infty\bigr)^{j+d} $ apparaissant dans 
l'int\'egrale de la formule 
\eqref{faltings}. Cela fera l'objet du th\'eor\`eme
\eqref{courant}. 
On  a tout d'abord la proposition suivante:

\begin{proposition}\label{formule} Pour tout
$i=1,\ldots,n-d-1$, on a
\begin{equation}\label{for}
\begin{split} 
h_{\overline{\mc}_\infty}\bigl(\X_{\A_i,
\beta}\bigr)&=\deg(R_i)\,h_{\overline{\mc}_\infty}\bigl(\X_{
\A_{i+1},\beta}\bigr)+\deg(X_{\A_{i+1},1})\log|N_K(\tau_{i})| \\
&+\sum_{\si:K\rightarrow \C}\int_{t\in\T^{d+i}(\C)}\log \frac{\bigl|
 Q_{i}\bigl(\ast_{\A_{i+1},\beta}(t)\bigr)
\bigr|_\si}{\max(|\ast_{\A_{i+1,\beta}}t |_\si )^{\deg(R_{i})}
 }\Bigl(dd^c
\log\max\bigl(|\ast_{\A_{i+1},\beta}(t)|_\si \bigr) \Bigr)^{d+i},
\end{split}
\end{equation}
et
\[
h_{\overline{\mc}_\infty}\bigl(\X_{\A_{n-d},
\beta}\bigr)=\log|N_K(\tau_{n-d})|+\sum_{\si:K\rightarrow \C}\int_{t\in\T^{n}(\C_\si)}\log \frac{\bigl|
 Q_{n-d}(t)
\bigr|_\si}{\max(1,|t_1|_\si,\ldots,|t_n|_\si) )^{\deg(R_{n-d})}
 }\Bigl(dd^c
\log\max\bigl(1,|t_1|_\si,\ldots,|t_n|_\si\bigr)  \Bigr)^{n}
\]
o\`u $|\ast_{\A_{i+1},\beta}(t)|_\si:=\bigl(|(\ast_{\A_{i+1},\beta}
(t))_0|_\si
\ldots,|(\ast_{\A_{i+1},\beta}(t))_n|_\si\bigr)$ avec
$\ast_{\A_{i+1},\beta}(t)=(\ast_{\A_{i+1},\beta}(t)_0,
\ldots,\ast_{\A_{i+1},\beta}(t)_n)$ (voir
\eqref{tor}).
\end{proposition}

\begin{proof}
Par la formule d\^ue \`a Faltings (voir \cite[(3.2.1) p. 949]{BoGS} et \cite[7,5.6.6]{Maillot}),
{\allowdisplaybreaks
\begin{align*}
 h_{\overline{\mc}_\infty}&(\X_{\A_i, \beta})=\deg(R_i)\,
 h_{{}_{\overline{\mathcal{O}(1)}_\infty}
}(\X_{\A_{i+1},
\beta})+\sum_{\si:K\rightarrow \C}
\int_{\X_{\A_{i+1},\beta}(\C) }\log \|s_i \|_{\si,\infty} 
c_1(
\overline{\mathcal{O}(1)}_\infty)^{d+i}\\
&=\deg(R_i)\,h_{{}_{\overline{\mathcal{O}(1)}_\infty} }(\X_{\A_{i
+1},
\beta})+ \sum_{\si:K\rightarrow \C}\int_{{}_{\X_{\A_{i+1},1 }(\C)}}\log \|s_i
(\beta\cdot x)
\|_{\si,\infty}\Bigl(dd^c \log
\max (|\beta\cdot x|_\si \Bigr)^{d+i}\\
&=\deg(R_i)\,h_{{}_{\overline{\mathcal{O}(1)}_\infty} }(\X_{
\A_{i+1},
\beta})+ \deg(X_{\A_{i+1},1})\log |N_K(\tau_{i})|\\
&+ \sum_{\si:K\rightarrow \C}\int_{\T^{d+i}(\C)}\log \frac{|Q_i(
\ast_{\A_{i+1,\beta}}t)|_\si}{\max\bigl(|\ast_{\A_{i+1},\beta}t |_\si 
)^{\deg(R_i)} }\Bigl(dd^c \log\max(|\ast_{\A_{i+1,\beta}}t|_\si
)\Bigr) ^{d+i}
\;\text{par}\, \eqref{negligeable}.
 \end{align*}}

La deuxi\`eme \'egalit\'e se d\'eduit de la m\^eme mani\`ere.
\end{proof}

Soit $\A':=\{a_1',\ldots, a_n' \}$ une famille de vecteurs de
$\Z^{d+1}$ telle que $L_{\A'}=\Z^{d+1}$ avec $d+1\leq n$ et 
$\beta=(1,\beta_1,\ldots,\beta_n)\in (\overline{\Q}^\ast)^{n+1} $. Le but de la suite consiste \`a \'etudier
 le courant suivant d\'efini sur $\T^{d+1}(\C)$
\[
 \Bigl(dd^c
\log\max(|\beta\cdot
t^{a'}| )\Bigr )^{d+1},
\]
o\`u $|\beta\cdot
t^{a'}|:=(1,|\beta_1 t^{a'_1}|,\ldots,|\beta_nt^{a'_n}|)$ pour 
tout $t\in \T^{d+1}$. Rappelons que ce courant appara\^it dans la formule
donnant la hauteur de $\X_{\A,\beta}$ (voir \eqref{formule}).\\

\noindent\textbf{\scshape{Notations} 1} :
 \begin{enumerate}
\item On note par $\mathrm{Log}$ l'application moment donn\'ee comme
suit
\begin{equation}\label{Log}
 \begin{aligned}
  \{x_0\neq 0\}=(\mathbb{C}^\ast)^{d+1}
&\stackrel{\mathrm{Log}}{\longrightarrow}\quad\mathbb{R}^{d+1}\\
   (z_1,\ldots,z_{d+1})\quad & \longrightarrow (\log|z_1|,\ldots,\log|z_{d+1}|).
 \end{aligned}
\end{equation}
\item Soit $f_k$  la fonction sur $\T^{d+1}(\C)$ d\'efinie par 
$f_k(t):=\beta_k t^{a'_k}, \forall t\in \T^{d+1}(\C)$ pour tout $k\in \{1, \ldots,n \}$.\\

 \item
 Dans $\R^{d+1}$, on pose
\begin{equation}\label{hyp}
H_{i,j}:=\Bigl\{ u\in \R^{d+1}\; \bigl|\, \left<a'_i-a'_j,u\right>=\log\frac{|\beta_j|}{|\beta_i|}\Bigr\}\quad \forall\, 1\leq i,j\leq n.
\end{equation}

On note par   $\mathscr{H} $  l'ensemble de sous-espaces   affines de $\R^{d+1}$ d\'efini comme suit: $H\in \mathscr{H}$, s'il existe $I$  un
sous-ensemble non-vide de $\bigl\{(i,j)\,\bigl|\; 1\leq i,j\leq d\bigr\}$ tel que $H=\cap_{(i,j)\in I}H_{i,j} $.\\

\item
Soit $S$ l'ensemble des points  de $\R^{d+1}$ d\'efini comme suit:  $s\in S$ si et seulement s'il existe un sous ensemble $I_s$ de $\{1,\ldots,n\}\times  \{1,\ldots,n\} $   tel que

\begin{equation}\label{s}
 \bigl\{s\bigr\}=\bigcap_{\tau\in I_s} H_\tau. 
\end{equation}

 \item On pose
\begin{equation}\label{xas}
\A'_{I_s}:=\bigl\{a'_i\in \A' \,| \;\exists\, a'_j,\; (i,j)\in I_s \bigr\},
\end{equation}
et on consid\`ere $X_{\A'_{I_s},1}$ la vari\'et\'e torique associ\'ee, au sens de Gelfand, Kapranov et Zelevinsky.\\

\item

Soit $s=(s_1,s_2,\ldots,s_{d+1})\in S $, on consid\`ere
\begin{equation}\label{cerlce}
\mathbf{S}_s:=\mathrm{Log}^{-1}(s)=\bigl\{z\in \mathbb{C}^{d+1}\;\bigl|\; |z_1|=|e^{s_1}|,\,\ldots,\, |z_{d+1}|=|e^{s_{d+1}}| \bigr\},
\end{equation}

et  $\delta_{\mathbf{S}_s} $, le courant int\'egration sur le polycercle $\mathbf{S}_s $.\\
\end{enumerate}

D'apr\`es la proposition $\eqref{formule}$, afin de calculer
$h_{\overline{\mc}_\infty}(\X_{\A,\beta}) $, il suffit de d\'eterminer
\[\Bigl(dd^c\log\max(|\beta\cdot t^{a'}|)\Bigr)^{d+1} \] C'est l'objet du th\'eor\`eme ci-dessous.\\

On consid\`ere sur $\T^{d+1}(\C)$ le courant suivant
\[
\omega_\beta:=dd^c \log\max\bigl(|\beta\cdot t^{a'}| \bigr)
\]
o\`u $|\beta\cdot t^{a'}|$ est par d\'efinition le vecteur 
$\bigl(1,|\beta_1t^{a_1'}|,\ldots,|\beta_n
t^{a'_n}|\bigr) $ pour tout $t\in \T^{d+1}(\C)$. Remarquons que
\[
\omega_\beta=(\ast_{\A',\beta})^{\ast}\bigl(c_1(\overline{\mathcal{O}(1)}_\infty)_{|_{X_{\A,\beta}
}}
\bigr),
\]
rappelons que $\ast_{\A',\beta}$ est le morphisme de 
$\T^{d+1}(\C)$ vers $\p^n$ introduit au d\'ebut de l'article en 
rempla\c{c}ant $\overline{\Q}$ par $\C$.

\begin{theorem}\label{courant}Soit $\A'$ comme avant. Sur $\T^{d+1}(\C)$, on a l'\'egalit\'e de courants suivante:
\[
                                              \omega_{\beta}^{d+1}=\sum_{s\in S}
\deg(X_{\A'_{I_s},1})
\delta_{{}_{\mathbf{S}_s}}.
                                             \]
\end{theorem}
Comme cons\'equence, nous disposons  d'un moyen de calcul, par 
r\'ecurrence, pour les hauteurs canoniques des sous-vari\'et\'es 
toriques $X_{\A,\beta}$ satisfaisant l'hypoth\`ese $\mathscr{A}$,  avec
 $\beta \in (F^\ast)^{n+1}$ o\`u $F$ est un corps de nombres dont
 l'anneau des entiers soit factoriel. C'est
l'objet du th\'eor\`eme ci-dessous. 
\begin{theorem}\label{laformulehot}Soit $n\in \N^\ast$ et $d\in 
\{1,\ldots,n-1\}$. Soit $\A:=\{a_1,\ldots,a_n\}$ une
sous-famille de $\Z^d$ de rang $d$ avec $L_\A\simeq \Z^d$. Soit $\beta\in 
(\overline{\Q}^\ast)^{n+1}$. On suppose que $X_{\A,\beta}$  v\'erifie l'hypoth\`ese 
$\mathscr{A}$ et que  $\beta \in (F^\ast)^{n+1}$ o\`u $F$ est un corps de nombres dont
 l'anneau des entiers soit factoriel.  On a, 
\begin{itemize}
\item Si $d\leq n-1$. On note $\A':=\A_2$ (voir \eqref{x3})  alors
il  existe $\tau\in K$
tel que
\begin{align*}
h_{\overline{\mc}_\infty}\bigl(\X_{\A,
\beta}\bigr)=&\deg(R_{1})\,h_{\overline{\mc}_\infty}\bigl(\X_{\A',
\beta}\bigr)+\deg(X_{\A',1})\log|N_K(\tau)|\\
&+\sum_{\si:K\rightarrow \C}\sum_{s_\si\in S_\si}
\deg(X_{\A'_{I_{s_\si}},1}) \int_{t\in\T^{d+1}(\C)}\log \frac{|Q_{1}(\beta
\cdot t^{a'}) |_\si}{\max(|\beta\cdot t^{a'} |_\si 
)^{\deg(R_1)} }
\delta_{\mathbf{S}_{s_\si}}.
\end{align*}
\item  Si $d=n-1$, alors il existe aussi $\tau\in K$ tel que
\begin{align*}
h_{\overline{\mc}_\infty}\bigl(\X_{\A,
\beta}\bigr)=\log|N_K(\tau)|
+\sum_{\si:K\rightarrow \C} \int_{t\in (\mathbb{S}^1)^{n}}\log |Q_{1}(t_1,\ldots,t_n
 ) |_\si\frac{dt_1\wedge \cdots dt_n}{t_1\ldots t_n}.
\end{align*}
\end{itemize}
\end{theorem}

\begin{proof}
Le th\'eor\`eme  r\'esulte des Th\'eor\`emes \eqref{formule} et  \eqref{courant}.
Les int\'egrales qui figurent dans les formules du th\'eor\`eme
seront explicit\'ees, voir Remarque \eqref{x7}.
\end{proof}

Afin de d\'emontrer le th\'eor\`eme \eqref{courant},
 on commence par \'etablir les lemmes d'alg\`ebre lin\'eaire suivants

\begin{lemma}\label{hyper} Soit $\mathbb{R}^n$ l'espace  r\'eel de dimension $n$. Soit $H$ un sous-espace affine de $\R^n$ de codimension $h$. S'il existe $\Omega $ un ensemble fini d'hyperplans affines tel que \[H=\bigcap_{\substack{{K\in \Omega}}} K,\]
alors, on peut extraire un sous ensemble $\Omega'\subset \Omega$ de cardinal $h$ tel que \[H=\bigcap_{\substack{{K\in \Omega'}}} K.\]
\end{lemma}

\begin{proof}

Il existe une matrice $A$ $(h\times n)$ de rang $h$, et $b\in \R^n$, tels que
\[
H=\{x\in \R^n \; |\; Ax=b\},
\]
Pour simplifier on peut supposer que $H$  et les \'el\'ements de $\Omega$ sont des espaces vectoriels \`a l'aide de l'application suivante :
\[
\begin{split}
H&\longrightarrow \ker A\\
x&\longrightarrow x-x_0
\end{split}
\]
o\`u $x_0$ est un \'el\'ement de $H$.
D\'emontrons le lemme par r\'ecurrence sur $h$; le cas $h=1$ est \'evident, donc  supposons que le r\'esultat est vrai pour $h\geq 1$. S'il existe $K_0\in \Omega$ tel que
\begin{equation*}
\mathrm{codim}(\bigcap_{\substack{K\in\Omega-\{K_0\} }} K)=h-1,
\end{equation*}
par suite, $\bigcap_{\substack{K\in\Omega-\{K_0\} }} K$ v\'erifie l'hypoth\`ese de r\'ecurrence, et on conclut en \'ecrivant $
H=K_0\cap (\bigcap_{\substack{K\in\Omega-\{K_0\} }}K )$.\\

Si maintenant, on a $ \forall\, K_0\in \Omega$, $ \mathrm{codim}(\bigcap_{\substack{K\in\Omega-\{K_0\} }} K)=h$, donc $H=
\bigcap_{\substack{K\in\Omega-\{K_0\}}}  K$, donc soit on est au premier cas, sinon on enl\`eve des \'el\'ements de $\Omega$, mais comme la codimension de l'intersection de $h $ espaces vectoriels de codimension $1$ est au plus $h$, cela termine la preuve du lemme.
\end{proof}
\begin{lemma}\label{in}
Soit $\R^n$ l'espace vectoriel r\'eel de dimension $n$, muni du produit scalaire standard qu'on notera  $<,>$.
 Soit $q\in \{1,\ldots, n-1\} $. On consid\`ere $q$ vecteurs $b_1,\ldots,b_{q}$   de $\R^n$
et  $q$ r\'eels $c_1,\ldots,c_{q}$  et  on pose  $H_i:=\{x\in \R^n\, |\, <b_i,x>=c_i\} $, $ \forall\, i\in \{1,\ldots,q\}$. Si
\begin{equation}\label{linear}
 \emptyset\neq\cap_{i=1}^{q-1}H_{i}\subseteq H_q,
\end{equation}
alors $b_q\in \mathrm{Vect}(b_1,\ldots,b_{q-1})$, o\`u $\mathrm{Vect}(b_1,\ldots,b_{q-1})$ est le sous-espace vectoriel de $\R^n$ engendr\'e par $b_1,\ldots,b_{q-1}$.
\end{lemma}

\begin{proof}
Si $\cap_{i=1}^{q-1}H_{i}\neq \emptyset $, alors  on peut se ramener au cas: $c_1=\cdots=c_q=0$ et supposer  que $\{b_1,,\ldots,b_{q-1}\}$ est libre. Posons $F:=\mathrm{Vect}(b_1,\ldots,b_{q})$.

 Soit $x\in F\cap (H_1\cap\cdots\cap H_{q-1})$, donc il existe
$\lambda_1,\ldots,\lambda_{q}\in \R^n$ tels que $x=\sum_{i=1}^{q}\lambda_i b_i$. Par \eqref{linear}, $<x,b_i>=0,\, \forall\, 1\leq i\leq q$, cela donne:\[
 0=\sum_{i=1}^q \lambda_i <b_i,b_j>, \, \forall\, 1\leq j\leq q.
\]

Si $b_q\notin \mathrm{Vect}(b_1,\ldots,b_{q-1})$, alors la matrice  suivante:
 \[(<b_i,b_j>)_{1\leq i,j\leq q}\] est inversible, on d\'eduit que
$\lambda_1=\cdots=\lambda_q=0$ donc $x=0$, cela implique que  $\dim(F+H_1\cap\cdots\cap H_{q-1})=q+n-(q-1)=n+1$, ce qui est impossible.

\end{proof}

\noindent\textbf{\scshape{Notations} 2} :
\begin{enumerate}
\item $\mathbb{N}_n=\{1,\ldots,n \} $.
 \item Pour tout $z\in \T^{d+1}(\C)$,  $\textbf{J}(z):=\{i\in  \mathbb{N}_n \,|\,
|f_i(z)|=\max(|f_1(z)|,\ldots,|f_n(z)|)  \}$.
\item Pour tout $z\in \T^{d+1}(\C)$,  $ \textbf{c}(z)=\mathrm{Card}(\textbf{J}(z))$.
\item $ \textbf{J}_{d+1}=\{z\in \T^{d+1}(\C)\,|\, \textbf{c}(z)\geq d+2 \} $. \\
\end{enumerate}

On note par $\textbf{L}$  le lieu de non-diff\'erentiabilit\'e de $\max (|f_1|,\ldots, |f_n| ) $ (on suppose que $\forall\, i\neq j$, $|f_i|\neq |f_j|$), c'est \`a dire:
\[
\textbf{L}=\bigl\{z\in \T^{d+1}(\C)\,|\, \exists i\neq j\in  \mathbb{N}_n, |f_i(z)|=|f_j(z)| \,\, \text{et}\, \max(|f_1(z)|,\ldots,|f_n(z)|)=|f_i(z)| \bigr\}.
\]
Soit $y\in \textbf{L}$ et posons $v:=\mathrm{Log}(y)$. \\

L'ensemble des $H\in \mathscr{H}$  contenant $v$ ordonn\'e par l'inclusion admet un plus petit \'el\'ement. Notons le par $H_v$ et soit
$h$ sa codimension dans $\R^{d+1}$.  Par le lemme  \eqref{hyper}, il existe $h+1$ hyperplans de la forme $H_{i,j}$ (cf.\eqref{hyper} ) tels que $H_v $ soit
leur intersection, cela est  \'equivalent \`a l'existence d'un sous-ensemble de $\{f_1,\ldots,f_n\}$
de cardinal $h+1$ qu'on suppose \'egal \`a
$\{f_1,\ldots,f_{h+1}\}$
telle que:
\[
 H_v=\mathrm{Log}\bigl\{|f_1|=\cdots=|f_{h+1}|\bigr\}.
\]

Sans perte de g\'en\'eralit\'e, on peut supposer que:\[
 \textbf{J}(y)=\bigl\{1,\ldots,h+1,h+2,\ldots,\textbf{c}(y)\bigr\}.
\]
Par continuit\'e, il existe $V_y$ un voisinage ouvert de $y$, tel que $\forall\, z\in V_y$, on a $\textbf{c}(z)=\textbf{c}(y) $ et $\textbf{J}(z)=\{1,\ldots,h+1,h+2,\ldots,\textbf{c}(y)\}$.

Posons:
\[
 b_j:=a_{j+1}-a_1,\, j\geq 1.
\]

S'il existe $i_0\notin\{1,\ldots,h+1\}$, tel que $|f_{i_0}(y)|=M(y)$ (ce qui est \'equivalent \`a $\textbf{c}(y)>h+1$), alors $v\in H_{1i_0}$ et donc par d\'efinition
de $H_v$, on a:

\[
 H_v\subset H_{1i_0}.
\]

Mais par le lemme \eqref{in}
\begin{equation}\label{remarque}
 b_{i_0}\in \mathrm{Vect}(b_1,\ldots,b_{h}).
\end{equation}


Par d\'efinition de $H_v$, $\{b_1,\ldots,b_h \}$ est libre. On choisit $b'_{h+1},\ldots,b'_{d+1}  $ $d+1-h$ vecteurs de $\Z^{d+1}$  tels que la famille
$\{b_1,\ldots,b_h,b'_{h+1},\ldots,b'_{d+1}\} $ soit libre et on consid\`ere l'application suivante:
\[
\begin{matrix}
                                \Phi: \C^{d+1} & \lra  & \C^{d+1} \\
t=(t_1,\ldots,t_{d+1}) & \longmapsto  & (\frac{\beta_2}{\beta_1}t^{b_1},\ldots,\frac{\beta_{h+1}}{\beta_{1}}t^{b_{h}},t^{b'_{h+1}},\ldots,t^{b'_{d+1}}).
\end{matrix}
\]
Puisque $\{b_1,\ldots,b_h,b'_{h+1},\ldots,b'_{d+1}\} $ est libre alors  $\Phi$ est un isomorphisme. \\

On introduit donc le changement de variables suivant:
\[
 x_i=\frac{\beta_i}{\beta_1}t^{b_i}(=\frac{\beta_i t^{a_i}}{\beta_1 t^{a_1}} ), \,1\leq i\leq h\quad\text{et}\quad x_i=t^{b'_i},\, h+1\leq i\leq d+1,
\]

Rappelons que sur $V_y$
\[\begin{split}
 M(t)&=\max(|f_1(t)|,\ldots,|f_{h+1}(t)|,|f_{h+2}(t)|,\ldots,|f_{\textbf{c}(y)}(t))\\
&=\max(|\beta_1 t^{a_1}|,\ldots,|\beta_{\textbf{c}(y)}t^{a_{\textbf{c}(y)}}| ),
\end{split}
\]

Sur l'ouvert $\Phi(V_y)$, on pose:
\[
 {\omega'}_\beta^{d+1}:=(dd^c M(\Phi^{-1}))^{d+1},
\]

Par \eqref{remarque}, il existe des $\lambda_{ij} \in \Z$ avec  $h+1\leq i\leq \textbf{c}(y)$ et $h+1\leq i\leq \textbf{c}(y),\, 1\leq j\leq h $
\[
 b_i=\sum_{j=1}^h\lambda_{ij}b_j,\quad h+1\leq i\leq \textbf{c}(y)
\]
  et
$\theta_{i}\in \overline{\Q}$ tels qu'on a sur $\Phi(V_y) $:\[\begin{aligned}
{\omega'}_\beta^{d+1}&=(dd^c\log\max(1,|x_1|,\ldots,|x_{h}|,|\theta_{h+1} \prod_{j=1}^{h} x_j^{\lambda_{(h+1)j}} | ,\ldots,|\theta_{\textbf{c}(y)}
\prod_{j=1}^{h} x_j^{\lambda_{\textbf{c}(y)j}}|))^{d+1}
\end{aligned}
\]


Montrons que  $\omega^{d+1}_\beta$ est   nul sur $V_y$ si $h\leq d$: Pour tout $p\geq 2$ soit  $T_p$ le courant positif d\'efini par la forme diff\'erentielle suivante:
\[{\omega'}_{\beta,p}^{d+1}=\Bigl(dd^c\log\bigl(1+|x_1|^p+\cdots+|x_{h}|^p+|\theta_{h+1} \prod_{j=1}^{h} x_j^{\lambda_{(h+1)j}} |^p +\cdots+|\theta_{\textbf{c}(y)}
\prod_{j=1}^{h} x_j^{\lambda_{\textbf{c}(y)j}}|^p\bigr)^\frac{1}{p}\Bigr)^{d+1}\]

Par la th\'eorie de Bedford et Taylor  (voir \cite{Bedford}), la suite $(T_p)_{p\geq 2}$ converge faiblement vers $\omega^{d+1}_\beta$. Remarquons que ${\omega'}_{\beta,p}^{d+1}$ est une forme de degr\'e $(d+1,d+1)$ qui est fonction de $x_1,\ldots,x_h$ mais comme $h\leq d$ alors cette forme est nulle. On conclut que ${\omega'}_{\beta}^{d+1}$ est le courant nul.\\

\noindent Le cas qui reste est  $h=d+1$, ce dernier  corresponds \`a $H_v=\{v\}$ , donc $v\in S$. \\

Comme $S$ est fini, alors $\forall\, s\in S$, $\exists V_s$ un voisinage ouvert de $\mathbf{S}_s$ tel que $V_s\cap \mathbf{S}_{s'}=\emptyset,\, \forall\, s'\neq s
$. On a donc montr\'e que
\[
\mathrm{Supp}(\omega_\beta^{d+1 })\subseteq \mathrm{Log}^{-1}(S).
\]
Plus pr\'ecis\'ement, on a \[ \mathrm{Supp}((\omega_\beta^{d+1})_{|_{V_s}})\subseteq \mathrm{Log}^{-1}(s)=\textbf{S}_s,\;\forall\, s\in S\]

Fixant un $s\in S$, il existe  $\Omega_s$  sous ensemble   de $\{f_1,\ldots, f_{n}\}$ de cardinal $d+1$ qui repr\'esente $\omega_\beta^{d+1}$ au voisinage de $\mathbf{S}_s$: $\omega_\beta^{d+1}=(dd^c \log \max_{f\in \Omega_s}|f|)^{d+1}$, au voisinage de $\mathbf{S}_s$.
On note par $\omega_{\beta,s}^{d+1}$ son extension \`a
$\T^{d+1}$ et par $X_{\A'_s,1}$ la sous-vari\'et\'e torique de $\p^n$ d\'efinie par $\Omega_s$. La raisonnement pr\'ec\'edent permet de conclure que cette vari\'et\'e  est de dimension $d+1$ et que

\begin{equation}\label{nonnul}
 \int_{\T^{d+1}(\C)}\omega_{\beta,s}^{d+1}\neq 0.
\end{equation}
Montrons qu'au voisinage de $\textbf{S}_s$, il existe une constante $C_s$ tel que
\[
\omega_{\beta,s}^{d+1}=C_s \cdot \delta_{\textbf{S}_s}.
\]

Commen\c{c}ons par montrer  que $\mathrm{Supp}(\omega_{\beta,s}^{d+1})=\mathbf{S}_s$. Si $\textbf{S}_s \backslash \mathrm{Supp}((\omega_{\beta,s}^{d+1})_{|_{V_s}}) \neq \emptyset$, c'est un ouvert de $\textbf{S}_s$;
Soit $D $ un ouvert connexe de $\textbf{S}_s \backslash \mathrm{Supp}((\omega_\beta^{d+1})_{|_{V_s}})$, on a
\begin{equation}\label{intD}
 \int_{\textbf{S}_s} \chi_{{}_D} \omega_{\beta,s}^{d+1}=0.
\end{equation}

$\chi_D$ la fonction caract\'eristique de $D$.\\

Remarquons que $\omega_{\beta,s}^{d+1}$ est invariant par rotation, en effet si $\theta $ est une rotation de $(\mathbb{S}^1)^{d+1}\subset \T^{d+1} $, ($\theta$ est un produit de
rotation sur chaque facteur $\mathbb{S}^1$), alors on a pour tout $\rho $ une fonction $\mathcal{C}^ \infty$ \`a support compact sur $\T^{d+1}$:

\[\begin{split}
\theta_\ast [\omega_{\beta,s}^{d+1}](\rho)&=\int_{\T^{d+1}} \theta^\ast(\rho)(x) \omega_{\beta,s}^{d+1}(x) \\
&= \int_{\T^{d+1}} \rho(x) ((\theta^{-1}) ^\ast\omega_{\beta,s}(x))^{d+1}\\
&=\int_{\T^{d+1}} \rho(x) \omega_{\beta,s}^{d+1}(x)\\
&=[\omega_{\beta,s}^{d+1}](\rho),
\end{split}
\]

donc le courant $ \omega_{\beta,s}^{d+1}$ est invariant par rotation, alors  $\eqref{intD} $ implique que $\int_{\textbf{S}_s} \omega_{\beta,s}^{d+1}=0$, ce qui
est contredit \eqref{nonnul}, on d\'eduit que

\[
 \mathrm{Supp}(\omega_{\beta}^{d+1})= \textbf{S}_s,
\]
au voisinage de $\textbf{S}_s$.

\begin{lemma}
Il existe $C_s=\deg(X_{\A'_{I_s},1})$, non nul, tel que
\[
 \omega_{\beta,s}^{d+1}=C_s \cdot \delta_{\mathbf{S}_s}.
\]
\end{lemma}

Par une dilatation du polycercle $\mathbf{S}_s $ en $(\mathbb{S}^1 )^{d+1} $, il suffit de d\'emontrer le lemme ci-dessous :
\begin{lemma} Soit $\mu$ une mesure positive sur $ (\mathbb{S}^1)^p$ invariante par l'action de $(\mathbb{S}^1)^p$,
alors il existe $c$ un r\'eel tel que
\[
\mu=c\cdot \delta_{(\mathbb{S}^1)^p}.
                                                                                                 \]

\end{lemma}

\begin{proof} C'est une cons\'equence de la th\'eorie des mesures de Haar sur les groupes de Lie
compacts.\\

Remarquons que $C_s=\deg(X_{\A'_{I_s},1}))$.

\end{proof}










On a donc terminer la preuve du th\'eor\`eme \eqref{courant}.

\begin{remarque}\label{x7}
\rm{ Soient $\A$ et $\A_2$ comme dans \eqref{x3}. On note
$\A':=\A_2$. Soit $\si:K\rightarrow \C$ une place \`a l'infini de $K$ et notons par
$S_\si$ l'ensemble associ\'ee \`a $X_{\A',\beta}$ (comme
dans Th\'eor\`eme
\eqref{courant}) pour la norme
$|\cdot|_\si$. 
   
   Soit $s_\si\in S_\si$. Par d\'efinition de l'ensemble $S_\si$ et
 par Lemme \eqref{hyper}, il existe $\{i_1\leq \ldots\leq i_{d+1} \}$ et 
 $\{j_1\leq \ldots\leq j_{d+1}\}$ 
deux  sous-ensembles de $\{0,1,\ldots,n \}$ tel que  
$\{s\}=\{u\in \R^{d+1}|\, \log |\beta_{i_1}|_\si+<a'_{i_1}, u>=
\log |\beta_{j_1}|_\si+<a'_{j_1}, u> =\cdots=
  \log |\beta_{i_{d+1}}|_\si+<a'_{i_{d+1}}, u>=\log
   |\beta_{i_{d+1}}|_\si
  +<a'_{j_{d+1}}, u>\}$\footnote{Les indices $\{i_1\leq \ldots\leq i_{d+1} \}$ et 
 $\{j_1\leq \ldots\leq j_{d+1}\}$ correspondent \`a 
 une sous matrice de $(a'_i-a'_j)_{1\leq i,j\leq n}$ de rang maximal, c-\`a-d $d+1$.}. Par un simple 
  argument d'alg\`ebre lin\'eaire,  il existe $d+1$ vecteurs de $\Q^{n+1}$ $v_{s,1},\ldots,v_{s,d+1}$, qui s'\'ecrivent en fonction 
  des $a_{i_\ast}$ et $a_{j_\ast}$, tels que
  \begin{equation}\label{x4}
  s_\si=(\log |\beta^{v_{s,1}}|_\si,\ldots,\log|\beta^{v_{s,d+1}}|_\si).
  \end{equation}
On note par $i(s)$ l'entier $i_1$. On a clairement  
  $\max(|\beta\cdot t^{a'} |_\si )=|\beta_{i(s)} 
e^{<s, 
a_{i(s)}'>}
|_\si$
pour tout   $t\in \mathbf{S}_s$.\\

  Explicitons l'int\'egrale de la 
formule du \eqref{laformulehot} lorsque $d\leq n-2$. On a 
\begin{align*}
\int_{t\in\T^{d+1}(\C)}\log \bigl|Q_1&(\beta
\cdot t^{a'}) \bigr|_\si\delta_{\mathbf{S}_{s_\si}}=\frac{1}{(2\pi)^{d+1}}\int_{t
\in \mathbf{S}_{s_\si}}\log
\bigl|(t^{a'})^{w_{+,1}}-
(t^{a'})^{w_{-,1}} \bigr|_\si\frac{dt_1\wedge\cdots\wedge dt_{d+1}}{t_1\cdots 
t_{d+1}}\\
=&\frac{1}{(2\pi)^{d+1}}\int_{t
\in \mathbf{S}_{s_\si}}\log
\bigl|(t^{a'})^{w_{-,1}} \bigr|_\si\frac{dt_1\wedge\cdots\wedge dt_{d+1}}{t_1\cdots 
t_{d+1}}+
\frac{1}{(2\pi)^{d+1}}\int_{t
\in \mathbf{S}_s}\log
\bigl|(t^{a'})^{w_{1}}-1\bigr|_\si\frac{dt_1\wedge\cdots\wedge dt_{d+1}}{t_1\cdots 
t_{d+1}}\\
=&\sum_{i=1}^n<s_\si,a'_i>w_{-,1,i}
+\frac{1}{(2\pi)^{d+1}}\int_{t
\in (\mathbb{S}^1)^{d+1}}\log
\bigl|((e^{s_\si}\cdot t)^{a'})^{w_{1}}-1
 \bigr|_\si\frac{dt_1\wedge\cdots\wedge dt_{d+1}}{t_1\cdots 
t_{d+1}}\\
=&\sum_{i=1}^n<s_\si,a'_i>w_{-,1,i}+ \sum_{i=1}^n<s_\si,a'_i>w_{1,i}+\log^+\bigl| e^{-\sum_{i=1}^n<s_\si,a'_i>w_{1,i}}\bigr|_\si\\
=&\sum_{i=1}^n<s_\si,a'_i>w_{-,1,i}+
\log^+\bigl| e^{\sum_{i=1}^n<s_\si,a'_i>w_{1,i}}\bigr|_\si.
\end{align*}
Notons qu'on a utilis\'e la formule de Jensen, avec 
$\log^+|\al|:=\max (0,\log |\al|)$ pour tout $\al \in \C$.  

Par suite,
\begin{equation}\label{x12}
\begin{split}
\int_{t\in\T^{d+1}(\C)}\log \frac{|Q_{1}(\beta
\cdot t^{a'}) |_\si}{\max(|\beta\cdot t^{a'} | )^{\deg(R_{1})} }\delta_{\mathbf{S}_{s_\si}}=&\sum_{i=1}^n<s_\si,a'_i>w_{-,1,i}+
\log^+\bigl| e^{\sum_{i=1}^n<s_\si,a'_i>w_{1,i}}\bigr|_\si\\
&-\deg(R_1)\log|\beta_{i(s)}|-\deg(R_1)<s_\si, a'_{i(s)}>.
\end{split}
\end{equation}
Notons, en particulier qu'
 il existe d'apr\`es  \eqref{x4}, un vecteur $v_{s_\si}\in \Q^n$
  tel que
le dernier terme s'\'ecrit $\log |\beta^{v_{s_\si}}
\beta_0^{-\sum_{j=1}^n 
v_{s_\si,j}}|_\si$.

On a donc,
\begin{align*}
h_{\overline{\mc}_\infty}&\bigl(\X_{\A,
\beta}\bigr)=\deg(Q_{1})\,h_{\overline{\mc}_\infty}\bigl(\X_{\A',
\beta}\bigr)+\deg(X_{\A',1})\log|N_K(\tau)|\\
&+ 
\sum_{\si:K\rightarrow \C}\sum_{s_\si\in S_\si}
\deg(X_{\A'_{I_{s_\si}},1}) \sum_{i=1}^n<s_\si,a'_i>w_{-,1,i}+\sum_{\si:K\rightarrow \C}\sum_{s_\si\in S_\si}
\deg(X_{\A'_{I_{s_\si}},1})\sum_{i=1}^n
\log^+\bigl| e^{\sum_{i=1}^n<s_\si,a'_i>w_{1,i}}
\bigr|_\si\\
&-\sum_{\si:K\rightarrow \C}\sum_{s_\si\in S_\si}
\deg(X_{\A'_{I_{s_\si}},1})
\deg(R_{1}) \log|\beta_{i(s_\si)}|_\si-\sum_{\si:K\rightarrow \C}
\sum_{s_\si\in S_\si}
\deg(X_{\A'_{I_{s_\si}},1})\deg(R_{1}) 
<s_\si, 
a_{i(s_\si)}'>,
\end{align*}
avec $\deg(X_{\A',1})=(d+1)!\mathrm{vol}_{\R^{d+1}}\mathrm{Conv}
(a'_1,\ldots,a'_{d+1})$ et 
$\deg(X_{\A'_{I_{s_\si}},1})=(d+1)!\mathrm{vol}_{\R^{d+1}}\mathrm{Conv}
(a_i|\exists j, (i,j)\in I_{s_\si})$ pour tout $s_\si\in S_\si$ et 
$\si:K\rightarrow \C.$
 
Si l'on note par $a_1^{(j)},\ldots,a_n^{(j)}$ les \'el\'ements
 de $\A_j$ pour $j=1,\ldots,n-d-1$ et $d\leq n-2$.
  Alors comme avant on a la 
  formule suivante 
{\allowdisplaybreaks
\begin{equation}\label{x18}
\begin{split}
h_{\overline{\mc}_\infty}&\bigl(\X_{\A_j,
\beta}\bigr)-\deg(R_{j})\,h_{\overline{\mc}_\infty}\bigl(\X_{\A_{j
+1},
\beta}\bigr)=\deg(X_{\A_{j+1},1})\log|N_K(\tau_j)|\\
&+ 
\sum_{\si:K\rightarrow \C}\sum_{s_\si\in S_{j,\si}}
\deg(X_{\A'_{j+1,I_{s_\si}},1}) \sum_{i=1}^n(s_\si\cdot a^{(j)}_i)w_{-,j,i}
+\sum_{\si:K\rightarrow \C}\sum_{s_\si\in S_{j,\si}}
\deg(X_{\A'_{j+1,I_{s_\si}},1})\sum_{i=1}^n
\log^+\bigl| e^{\sum_{i=1}^n(s_\si\cdot a^{(j)}_i)w_{j,i}}
\bigr|_\si\\
&-\sum_{\si:K\rightarrow \C}\sum_{s_\si\in S_{j,\si}}
\deg(X_{\A'_{j+1,I_{s_\si}},1})
\deg(R_{j}) \log|\beta_{i(s)}|_\si-\sum_{\si:K\rightarrow \C}
\sum_{s_\si\in S_{j,\si}}
\deg(X_{\A'_{j+1,I_{s_\si}},1})\deg(R_{j}) 
<s_\si, 
a_{i(s_\si)}^{(j)}>,
\end{split}
\end{equation}
}
avec $\deg(X_{\A_{j+1},1})=(d+j+1)!\mathrm{vol}_{\R^{d++j+1
}}\mathrm{Conv}
(a^{(j+1)}_1,\ldots,a^{(j+1)}_{d+1})$ et pour tout $s_\si\in S_{j+1,\si}
$,
$\deg(X_{\A'_{j+1,I_{s_\si}},1})=(d+j+1)!\mathrm{vol}_{\R^{d+j+1}}\mathrm{Conv}
(a_i^{(j+1)}|\exists k, (i,k)\in I_{s_\si})$ pour tout $\si:K\rightarrow \C$. \\

Si $d=n-1$, alors d'apr\`es ce qui pr\'ec\`ede ou par \cite[proposition 7.2.1]{Maillot}, on a
\begin{equation}\label{ee1} 
\begin{split}
h_{\overline{\mc}_\infty}\bigl(\X_{\A_{n-d},
\beta}\bigr)=&\log |N_K(\tau_1)|+ \sum_{\si:K\rightarrow \C}\int_{t\in(\mathbb{S}^1
)^{n}}\log|Q_{n-d}(
 t) |_\si \frac{dt_1\wedge\cdots dt_n}{t_1\cdots t_n}\\
 =&\log |N_K(\tau_1)|
 +
 \log|N_K(\beta^{-w_{+,n-d}})|+\sum_{\si:K\rightarrow \C}\log^+|\beta^
 {w_{n-d}}|_\si.
\end{split}
\end{equation}
}

\end{remarque}
\begin{Corollaire}\label{x6}
En gardant les m\^emes hypoth\`eses que dans \eqref{laformulehot}, alors
il existe $u_{\A}\in \N^{n-d}$ et $(v_{\A,\si,i})_{
\substack{\si:K\rightarrow \C\\
i=1,\ldots,n}}$ une sous famille de $\Q^n$ tels que
\[
h_{\overline{\mc}_\infty}\bigl(\X_{\A,
\beta}\bigr)=\sum_{i=1}^{n-d}u_{\A,i}\log|N_K(\tau_i)|+                                                                                                                                                                                                                                                                                                                                                                                                                                                                                                                                                                                                                                                                                                                                                                                                                                                                                                                                                                                                                                                                                                                                                                                                                                                                                                                                                                                                                                                                                                                                                                                                                                                                                                                                                                                                                                                                                                                                                                                                                                                                                                                                                              
\sum_{\substack{\si:K\rightarrow \C\\i=1,\ldots,n}}v_{\A,\si,i}\log
\Bigl|\frac{\beta_i}{\beta_0}\Bigr|_\si.
\]
On a,
\[
h_{\overline{\mc}_\infty}\bigl(\X_{\A,
\beta}\bigr)\in \log\bigl(\overline{\Q}\cap \R_{>0} \bigr).
\]
Donc, si $h_{\overline{\mc}_\infty}\bigl(\X_{\A, \beta}\bigr)\neq 0$,  alors c'est un nombre transcendant.
\end{Corollaire}

\begin{proof}
Ce r\'esultat est une cons\'equence de \eqref{laformulehot} et
de \eqref{x7}. Par la formule du corollaire, on voit clairement que 
$h_{\overline{\mc}_\infty}\bigl(\X_{\A, \beta}\bigr)
=\log(\gamma)$, avec $\gamma \in \overline{\Q}\cap \R_{>0}$. 
Par le th\'eor\`eme de Baker (voir par exemple \cite{Waldschmidt}), $h_{\overline{\mc}_\infty}\bigl(\X_{\A, \beta}\bigr)$ est transcendant si $h_{\overline{\mc}_\infty}\bigl(\X_{\A, \beta}\bigr)\neq 0$.

\end{proof}

\begin{example}

\begin{enumerate}
\item Soit $\beta\in (\Q^\ast)^{n+1}$
et $X_{A,\beta}$ est une hypersurface torique int\`egre de
$\p^n$. Soit $w=(w_1,\ldots,w_n)\in \Z^n$  tel que
$w\in \ker({}^t A)$ o\`u $A$ est la matrice dont les lignes
sont les \'el\'ements de $\A$. On choisit $w$ tel que 
$\gcd(w_1,\ldots,w_n)=1$. Alors la vari\'et\'e
$X_{\A,\beta}$ est d\'efinie par $\beta^{-w_+}x^w-\beta^{-w_-}$. Soient $a,b\in \Z$ avec $\gcd(a,b)=1$ tels que
$\beta^{w}=\frac{a}{b}$, (voir notation \eqref{x16}). D'apr\`es \eqref{ee1}, on a
\begin{equation}\label{112}
h_{\overline{\mathcal{O}(1)}_\infty}(\X_{\A,\beta})=\log|b \beta^{w_+}|
+\log|\beta^{-w_+}|+\log^+|\beta^w|=\log\max(|a|,|b|).
\end{equation}

\item
Soit $c=(c_1,c_2,c_3)\in (\Q^\ast)^3$. Soit
 $\A$ le singleton $(1,-1,3)$. 
On consid\`ere la vari\'ete torique associ\'ee $X_{\A,c}$, c'est une
courbe  dans $\p^3$. On consid\`ere 
  $u_1:=(-2,1,1,0)$, $u_2:=(0,2,-1,-1)$,  
  $w_1:=(1,1,0)$ et $w_2:=(2,-1,-1)$. 
 Soit $M$ la matrice suivante
\[
M=\begin{bmatrix}
-2&1&1&0\\
0&2&-1&-1
\end{bmatrix}
\]
  C'est une matrice qui v\'erifie les conditions de la d\'efinition
  \eqref{mmd}. Donc
  d'apr\`es la proposition \eqref{modeleintegre},
$R_1=c_1^{-1}c_2^{-1}T^{u_{+,1}}-T^{u_{-,1}}$ et $R_2=c_1^{-2}T^{u_{+,2}}-c_2^{-1}c_3^{-1}T^{u_{-,2}}$  d\'efinissent une vari\'et\'e torique qui satisfait D\'efinition
\eqref{condA} et le r\'esultat du Lemme \eqref{x5}. On v\'erifie que cette vari\'et\'e co\"incide avec $X_{\A,c}$. On se propose de donner
une formule pour $h_{\overline{\mathcal{O}(1)}_\infty}(\X_{\A,c})$
pour tout $c\in (\Q^\ast)^3$. En gardant les m\^emes hypoth\`eses
et notations, on a la formule suivante pour la hauteur canonique de $\X_{\A,c}$

\begin{theorem}\label{xample}
Soient $a,b\in \Q$ tels que $\frac{a}{b}=\beta^{w_1}(=c_1c_2)$ et $\gcd(a,b)=1$. 
 Si $|c_1c_2|<1$, alors 
 \[
 h_{{}_{\overline{\mathcal{O}(1)}_\infty}}(\X_{\A,c})=2\log\max(|a|,|b|)+ 2\log |\tau_2|-2\log|c_1|+2\log|c_2|+\log^+|c_1^3c_3^{-1}|+\log^+|
c_2^{-3}c_3^{-1}|
 \]
 Si $|c_1c_2|\geq 1$, alors
\[
h_{{}_{\overline{\mathcal{O}(1)}_\infty}}(\X_{\A,c})=2\log\max(|a|,|b|)+2\log
 |\tau_2|-4\log|c_1|+2\log^+|c_1^{2}c_2^{-1}c_3^{-1}|.
\] 
 En particulier, si $c_1,c_2,c_3\in \Z$ avec 
 $\gcd(c_i,c_j)=1$ pour tout $i\neq j\in \{1,2,3\}$ alors, on a
 \[
 h_{{}_{\overline{\mathcal{O}(1)}_\infty}}(\X_{\A,c})=2\log
 |c_1c_2|+\log\max(|c_1^2|,|c_2c_3|).
 \]

\end{theorem}
D\'emontrons ce r\'esulat.
On pose $A_1:=A$. Par construction, on observe que $\{(1,-1,0),(0,0,1)\}$ est une $\Z$-base 
de  $ \{x\in \Z^3|\,
<x,w_1>=0 \}$. On pose $A_2$ la matrice de taille $3\times 2$ suivante
\[
A_2=\begin{bmatrix}
1&0\\
-1&0\\
0&1
\end{bmatrix}
\]
 On note par $a'_1,a_2',a_3'$ les lignes 
de $A_2$ et on pose $a_0'=(0,0)$ et $c_0=1$ et on note
$\A_2=\{a_0',a_1',a_2',a_3'\}$. D'apr\`es le 
th\'eor\`eme \eqref{courant}, on sait que 
le courant $\bigl(dd^c\max(1,|c_1t^{a'_1}|,|c_2t^{a'_2}|,
|c_3t^{a'_3}|)\bigr)^2$ est d\'etermin\'e en termes d'un ensemble
fini de points $S$ dans $\R^2$. Rappelons que 
$s\in S$, s'il existe $i_0,i_1,i_2,i_3\in \{0,1,2,3\}$ tels que 
$|c_{i_0}(e^{s})^{a'_{i_0}}|=|c_{i_2}(e^{s})^{a'_{i_1}}|=|c_{i_2}(e^{s})^{a'_{i_2}}|=|c_{i_3}(e^{s})^{a'_{i_3}}|$ et 
$|c_j (e^s)^{a_j'}|<|c_{i_0}(e^{s})^{a_{i_0}}|$ pour tout
$j\in \{0,1,2,3\}\setminus \{i_0,i_1,i_2,i_3 \}$. 

On pose $u_i:=\log|t_i|$ et
$\al_i:=\log|c_i|$ pour $i=1,2$. On dispose de 5 situations
 \`a \'etudier:

\begin{enumerate}
 
\item Le cas $1=|c_1t_1|=|c_2 t_1^{-1}|$ et $|c_3t_2|<1$. 
Cela donne $u_1=-\al$, $u_1=\al_2$. 
Si $\al_2\neq -\al_1$ alors l'ensemble des solutions est vide. Si $\al_2=
 -\al_1$, alors l'ensemble des solutions $(u_1,u_2)$ correspond
 \`a une demi-droite. Donc, $\bigl(dd^c\max(1,|c_1t^{a'_1}|,|c_2t^{a'_2}|,
|c_3t^{a'_3}|)\bigr)^2$ est nulle au voisinage de tout point 
$(t_1,t_2)$ tel que $(u_1,u_2)$ soit proche de cette demi-droite (voir
la preuve du th\'eor\`eme \eqref{courant}).
\begin{enumerate}
\item Plus g\'en\'eralement, soient 
$i_1,i_2\in \{0,1,2,3\}$. On consid\`ere le syst\`eme suivant $|c_{i_1}t^{a_{i_1}'}|=|c_{i_2}t^{a'_{i_2}}|>|c_j t^{a_j'}|$ pour tout $j\neq i_1,i_2$. Alors, Donc comme avant que l'intersection
de $S$ avec l'ensemble des solutions de ce syst\`eme est vide. on v\'erifie que l'intersection de $S$ avec 
l'ensembe des solutions de ce syst\`eme est vide.
\end{enumerate}

\item Le cas $1=|c_1t_1|=|c_3t_2|$ et $|c_2t_1^{-1}|<1$. On a alors
$u_1=-\al_1$, $u_2=-\al_3$ et $\al_2<u_1$. Donc, si 
$\al_2\leq -\al_1$, alors on dispose d'une solution unique 
$s=(-\al_1,-\al_3)\in S$.
\item $1=|c_2t_1^{-1}|=|c_3t_2|$ et $|c_1t_1|<1$. On v\'erifie qu'on
dispose d'une unique solution $s=(\al_2,-\al_3)\in S$ si et seulement
si $\al_2\leq -\al_1$.
\item Le cas $|c_1t_1|=|c_2 t_1^{-1}|=|c_3t_2|$ et $|c_1 t_1|<1$.
On v\'erifie qu'on a une unique solution $s=(\frac{\al_2}{2}-\frac{\al_1}{2}, -\frac{\al_2}{2}+\frac{\al_1}{2}-\al_3)$ si
et seulement si $\al_2\geq -\al_1$.

\item $1=|c_1t_1|=|c_2 t_1^{-1}|=|c_3t_2|$. On a une unique solution
$s=(-\al_1,-\al_3)$ si et seulement si $-\al_1=\al_2$.
\end{enumerate}

R\'ecapitulons; On distingue 2 situations:
\begin{itemize}
\item Lorsque $\al_1+\al_2<0$ c-\`a-d $|c_1c_2|<1$, alors $S=\{s_1,s_2 \}$ avec
$s_1=(-\al_1,-\al_3)$ et $s_2=(\al_2,-\al_3)$, et
$\deg(X_{\A'_{I_{s_1}}})=\deg(X_{\A'_{I_{s_2}}})=1$ puisque
$\A'_{I_{s_1}}=\{a'_1,a_3'\}$ et $\A'_{I_{s_3}}=\{a_2',a_3'\}$.
\item $\al_1+\al_2\geq 0$ alors $S=\{s_1\}$ avec
$s_1=(\frac{\al_2}{2}-\frac{\al_1}{2},\frac{\al_1}{2}+\frac{\al_2}{2}-\al_3)$, et $\deg(X_{\A'_{I_{s_1}}})=2$ puisque
$\A'_{I_{s_1}}=\{a_2'-a'_1,a_3'-a_1'\}$.
\end{itemize}

On a $\ast_{\A_2,c}(t_1,t_2)=(1,c_1t_1,c_2t_1^{-1},c_2t_2), \forall 
t_1,t_2\in \C^\ast$ (voir les notations \eqref{tor}), donc
$R_2\bigl(\ast_{\A_2,c}(t_1,t_2)\bigr)=t_1^2-t_1^{-1}t_2$. On a alors,
\begin{align*}
\int_{{}_{\X_{\A_2,\beta}(\C)}}\log\bigl\|&R_2\bigr\|_{\infty}c_1\bigl(
\overline{\mc}_\infty\bigr)^{2}\\
&=\int_{{}_{\T^2(\C)}}\log\frac{|t_1^2-t_1^{-1}t_2
\bigr|}{\max(1,|c_1t_1|,|c_2t_1^{-1}|,|c_3t_2|)^2}
\bigl(dd^c \log \max(1,|c_1t_1|,|c_2 
t_1^{-1}|, |c_3 t_2|) \bigr)^2.
\end{align*}

Si l'on  note $s:=(r_1,r_2)\in S$, alors  par
la formule de Jensen, on a
\[
\int_{{}_{\T^2(\C)}}\log\frac{|t_1^2-t_1^{-1}t_2
\bigr|}{\max(1,|c_1t_1|,|c_2t_1^{-1}|,|c_3t_2|)^2}\delta_{
\mathbf{S}_{s}}=2\log|r_{1}|+\log^+|r_{1}^{-3}r_{2}|
-2\log \max(1,|c_1r_{1}|,|c_2r_{1}^{-1}|,|c_3r_{2}|).
\]
Donc, si $|c_1c_2|<1$ (c-\`a-d $\al_1+\al_2<0$) alors
\begin{align*}
\int_{{}_{\T^2(\C)}}&\log\frac{|t_1^2-
t_1^{-1}t_2
\bigr|}{\max(1,|c_1t_1|,|c_2t_1^{-1}|,|c_3t_2|)^2}
\bigl(dd^c \log \max(1,|c_1t_1|,|c_2 
t_1^{-1}|, |c_3 t_2|) \bigr)^2\\
&=\sum_{i=1}^{2}\deg(X_{\A'_{I_{s_i}},1}) \int_{{}_{\T^2(\C)}}\log\frac{|t_1^2-t_1^{-1}t_2
\bigr|}{\max(1,|c_1t_1|,|c_2t_1^{-1}|,|c_3t_2|)^2}\delta_{
\mathbf{S}_{s_i}}\\
&=-2\log|c_1|+\log^+|c_1^3c_3^{-1}|-2\log^+|c_1c_2|+2\log|c_2|+\log^+|
c_2^{-3}c_3^{-1}|-2\log^+|c_1c_2|\\
&=-2\log|c_1|+2\log|c_2|+\log^+|c_1^3c_3^{-1}|+\log^+|
c_2^{-3}c_3^{-1}|.
\end{align*}
D'apr\`es \eqref{112},
\[
h_{\overline{\mathcal{O}(1)}_\infty}(\X_{\A_2,c})=\log\max(|a|,|b|),
\]
avec $\frac{a}{b}=\beta^{w_1}(=c_1c_2)$ et $\gcd(a,b)=1$.
 
 Par suite,
 \begin{align*}
 h_{\overline{\mathcal{O}(1)}_\infty}(\X_{\A,c})&=2h_{\overline{\mathcal{O}(1)}_\infty}(\X_{\A_2,c})+
 2\log |\tau_2|-2\log|c_1|+2\log|c_2|+\log^+|c_1^3c_3^{-1}|+\log^+|
c_2^{-3}c_3^{-1}|\\
&=2\log\max(|a|,|b|)+ 2\log |\tau_2|-2\log|c_1|+2\log|c_2|+\log^+|c_1^3c_3^{-1}|+\log^+|
c_2^{-3}c_3^{-1}|
 \end{align*}
 
 Si $|c_1c_2|\geq 1$. On a
 \begin{align*}
\int_{{}_{\T^2(\C)}}&\log\frac{|t_1^2-
t_1^{-1}t_2
\bigr|}{\max(1,|c_1t_1|,|c_2t_1^{-1}|,|c_3t_2|)^2}
\bigl(dd^c \log \max(1,|c_1t_1|,|c_2 
t_1^{-1}|, |c_3 t_2|) \bigr)^2\\
&=\deg(X_{\A'_{I_{s_1}},1}) \int_{{}_{\T^2(\C)}}\log\frac{|t_1^2-t_1^{-1}t_2
\bigr|}{\max(1,|c_1t_1|,|c_2t_1^{-1}|,|c_3t_2|)^2}\delta_{
\mathbf{S}_{s_1}}\\
&=-2\log |c_1|+2\log|c_2|+2\log^+|c_1^{2}c_2^{-1}c_3^{-1}|-2\log^+| c_1c_2|\\
&=-4\log|c_1|+2\log^+|c_1^{2}c_2^{-1}c_3^{-1}|.
\end{align*}
 Donc, de la m\^eme mani\`ere pr\'ec\'edente on trouve que
 \[
 h_{\overline{\mathcal{O}(1)}_\infty}(\X_{\A,c})=2\log\max(|a|,|b|)+2\log
 |\tau_2|-4\log|c_1|+2\log^+|c_1^{2}c_2^{-1}c_3^{-1}|,
 \]
 o\`u $a,b\in \Q$  tels que $\frac{a}{b}=\beta^{w_1}(=c_1c_2)$ et $\gcd(a,b)=1$.
 
 \end{enumerate}

La derni\`ere formule du th\'eor\`eme se d\'eduit ais\'ement
de la deuxi\`eme formule.
\end{example}

\begin{remarque}\label{x300}
Soit $c\in \N_{\geq 2}$. On consid\`ere la matrice suivante de
taille $8\times 10$
\[
M=\begin{bmatrix}
-2 & 1 & 1&0 & 0 &0&0 &0&0&0\\
-2c & 0 & 2c-1 &1&0&0&0&0&0&0\\
-3&0&2&0&1&0&0&0&0&0\\
-2c-1&0&2c &0&0&0&1&0&0&0\\
1&0&-2&0&0&0&0&1&0&0\\
-2c+1&0&2c-2&0&0&0&0&0&1&0\\
-4c+1&0&4c-2&0&0&0&0&0&0&1\\
-1&0&0&0&0&1&0&0&0&0
\end{bmatrix}
\]
$M$ definit une sous-vari\'et\'e dans $\mathbb{P}^9$ 
(pour la construction voir \eqref{matrice}). 
On observe que  M est  mixte avec  contenu \'egal \`a 1, mais
$M$ n'est dominante (voir les d\'efinitions de la section \eqref{x100}). En effet, la sous-matrice form\'ee par les 8 premi\`eres colonnes est mixte.
On consid\`ere la vari\'et\'e $X_{\A,1}\subset \mathbb{P}^9$ donn\'ee par
$\A=\{1,-1,2c-1,2,2c,-2,2c-2,4c-2\}\subset \mathbb{Z}.$ 
Les lignes de $M$ forment  une $\mathbb{Z}$-base pour 
$\ker(1,-1,2c-1,2,2c,-2,2c-2,4c-2 )$. 
Mais $X_{\A,1}$ ne v\'erifie pas la  condition
$\mathscr{A}$. Soit $f$ le morphisme suivant, 

\[
f:\mathbb{P}^3\rightarrow \mathbb{P}^9, \,[x_0:x_1:x_2:x_3]\mapsto 
[x_ix_j, 0\leq i,j\leq 3].
\]
On v\'erifie que $X_{\A,1}$ est l'image de 
$X_{\A_0,1}\subset \mathbb{P}^3$ par
$f$, o\`u $X_{\A_0,1}$ est une  vari\'et\'e torique avec  $\A_0=\{1,-1,2c-1\}$ et que
la matrice suivante 
\begin{equation*}
\begin{bmatrix}
0 & c & 1-c
& -1 \\
-2 & 1 & 1
& 0 
\end{bmatrix}
\end{equation*}
 v\'erifie D\'efinition \eqref{mmd} et qu'elle d\'efinit  
 $X_{\A,c}$ (voir les d\'efinitions et la construction de la
  section \eqref{x100}). Donc cette vari\'et\'e satisfait l'hypoth\`ese
 $\mathscr{A}$ d'apr\`es Proposition 
 \eqref{modeleintegre}. Soit $\beta'\in (\overline{\mathbb{Q}}^\ast)^4$, et  $\beta:=f(\beta')$. Alors
on a (par \cite[Th\'eor\`eme 5.5.6]{Maillot}) \[h_{ \overline{\mathcal{O}(1)}_\infty}( X_{\A,\beta})=h_{ \overline{\mathcal{O}(1)}_\infty}(f_\ast X_{\A_0,\beta'})=
h_{f^\ast \overline{\mathcal{O}(1)}_\infty}(X_{\A_0,\beta'})=
h_{ \overline{\mathcal{O}(1)}_\infty}(X_{\A_0,\beta'}).
\]
Donc, si l'on suppose en plus que  $\beta' \in (F^\ast)^{n+1}$ o\`u $F$ est un corps de nombres dont
 l'anneau des entiers soit factoriel 
  par rapport \`a $\A_0$, alors on dispose
d'une formule pour la hauteur canonique de $ \X_{\A,\beta}$
similaire au cas \'etudi\'e pr\'ec\'edemment. 

Cette observation sugg\`ere la g\'en\'eralisation suivante:
Soient $d\leq N$ deux entiers positifs. Soit $\A=
\{a_1,\ldots,a_N\}$ un sous-ensemble de $\Z^{d}$ avec $L_{\A}\simeq\Z^{d}$ et  soit $A$ la 
matrice associ\'ee. On suppose que $A=C\cdot B$ avec
$B$ (resp. $C$) une matrice de taille $N\times n$ (resp. 
$n\times d$) \`a 
coefficients dans $\Z$.
 On note par $\mathcal{B}$ (resp. $\mathcal{C}$) 
l'ensemble form\'e par les lignes de $B$ (resp. $C$). On suppose
que $\mathcal{C}$ d\'efini un morphisme monomial de $\p^n$ 
vers $\p^N$ qu'on note par 
$f$. Soit $\beta\in \p^N(\overline{\Q})$ 
telque $\beta=f(\beta')$ avec $\beta'\in \p^n(\overline{\Q})$.
Supposons que $\mathcal{B}$ et $\beta'$ v\'erifient les conditions $\mathscr{A}$ et que  $\beta' \in (F^\ast)^{n+1}$ o\`u $F$ est un corps de nombres dont
 l'anneau des entiers soit factoriel. Alors, on dispose
d'une formule pour la hauteur canonique pour 
$\X_{\A,\beta}$ similaire au cas pr\'ec\'edent.

\end{remarque}
\section{Un r\'esultat sur les m\'etriques admissibles}

Dans cette section on va montrer un r\'esultat technique \`a savoir le th\'eor\`eme \eqref{negligeable}.
Rappelons qu'on l'a  utilis\'e dans la preuve de la proposition \eqref{formule}. \\

\begin{definition}
Soit $X$ une vari\'et\'e projective complexe de dimension $d$ et $(L,\vc)$ un fibr\'e en droites
holomorphe muni d'une
m\'etrique
hermitienne continue sur $X$. On dit que $\vc$ est admissible s'il existe $(\vc_k)_{k\in \N}$une
suite de m\'etriques hermitiennes positives de classe $\cl$, qui converge uniform\'ement vers
$\vc$ sur $L$.
\end{definition}
On d\'efinit alors le premier courant de Chern associ\'e \`a $(L,\vc)$, le courant $c_1(L,\vc)$ d\'efini
localement par:
\[
 dd^c\bigl(-\log \|s\|^2 \bigr),
\]
o\`u$s$ est une section locale holomorphe non nulle de $L$. D'apr\`es \cite{Demailly},   on peut
aussi consid\'erer les
puissances  du courant $c_1(L,\vc)$.

Il est bien connu que lorsque la m\'etrique est $\cl$ et strictement
positive, alors le support de $c_1(L,\vc)^d$ est Zariski-dense. En fait, on montre qu'il est \'egal \`a $X$. Le th\'eor\`eme \eqref{negligeable} a pour
but d'\'etablir le m\^eme r\'esultat pour les fibr\'es en droites admissibles. Commen\c{c}ons d'abord par l'exemple suivant; si on
consid\`ere les m\'etriques canoniques sur une vari\'et\'e torique lisse $X$, alors on montre dans
\cite{Maillot}, que ces derni\`eres sont admissibles et on calcule explicitement les diff\'erentes
puissances des courants de Chern associ\'es, en particulier  le support des puissances maximal est
le sous-tore compact $(\mathbb{S}^1)^d$, qui est dense pour la topologie de Zariski dans $X$. En effet, si
  $P=\sum_\nu a_\nu x^\nu$ est un polyn\^ome \`a $d$ variables, nul sur $(\mathbb{S}^1)^d$, c'est \`a dire
$\sum_{\nu}a_\nu e^{i<\nu, \theta>}=0$ pour tout $\theta\in \R^d$. Alors, par orthogonalit\'e des fonctions trigonom\'etriques, on
trouve que $a_\nu=0$.

Un autre exemple est celui d'une m\'etrique $\vc$ qui provient d'un syst\`eme dynamique sur $(X,L,f)$, alors $\vc$ est admissible,
voir \cite{Zhang}.  On montre que le support coïncide avec l'ensemble de Julia associ\'e \`a $f$,
voir \cite{dynamicsibony}, et
que ce dernier contient les points p\'eriodiques de $f$ qu'on montre qu'il est Zariski-dense dans
$X$, voir \cite[th\'eor\`eme 5.1]{Fakhruddin}. Par cons\'equent, le support de $c_1(L,\vc)^{\dim(X)}$ est Zariki-dense. Ces r\'esultats sont une
combinaison des th\'eor\`emes
ci-dessous:

Rappelons le th\'eor\`eme suivant sur la densit\'e des points p\'eriodiques d\'efinis par une dynamique alg\'ebrique:
\begin{theorem}\label{fakhruddin}
Soit $X$ une vari\'et\'e projective sur un corps alg\'ebriquement clos $k$, $\phi:X\rightarrow X$ un morphisme dominant et $L$ un fibr\'e en droites sur $X$ tels que $\phi^\ast L\otimes L$ soit ample. Alors le sous-ensemble de $X(k)$ form\'e des points p\'eriodiques de $\phi$ est Zariski-dense dans $X$.
\end{theorem}
\begin{proof}
Voir \cite[th\'eor\`eme 5.1]{Fakhruddin}.
\end{proof}

Si $\vc_f$ la m\'etrique invariante par $f$, si l'on pose

\[
\mu_f=\frac{1}{\deg_L(X)}c_1(L,\vc_f)^{\dim X}
\]
Lorsque $k=\C$. Pour tout $m\geq 1$, on note par $P_m$ l'ensemble des points dans $X(\C)$ p\'eriodiques de p\'eriode $m$ compt\'es avec multiplicit\'es. Consid\'erons
\[
\delta_{P_m}=\frac{1}{\mathrm{\mathrm{Card}}(P_m)}\sum_{z\in P_m}\delta_z
\]

on a le th\'eor\`eme suivant:
\begin{theorem}
Si $X=\p^n$, la suite $\delta_{P_m}$ converges faiblement vers $\mu_f$.
\end{theorem}
\begin{proof}
Voir \cite{BD99}.
\end{proof}

Soit $X$ un sous-ensemble analytique de $\p^n$ de dimension $d$. Soit $\overline{L}=(L,\vc)$ un fibr\'e en droites admissible sur $\p^n$.  On a
\begin{theorem}\label{negligeable}
Soit $X$ une vari\'et\'e projective complexe  de dimension $d$. Si $\overline{L}$ est
un fibr\'e en droites admissible tel que $c_1(L)^d\neq 0$, alors le support du courant $c_1(\overline{L})^d$ est
Zariski-dense dans $X$.

Si $H$ est une hypersurface de $X$, alors
\[
\int_{ X} \rho c_1(\overline{L})^d=\int_{X\setminus H}\rho c_1(\overline{L})^d,
\]
 pour toute fonction $\rho $ continue sur $X$.
\end{theorem}

\begin{proof}

Par l'absurde, supposons qu'il existe une hypersurface $H$ de $X$ telle que
\[\mathrm{\mathrm{Supp}}(c_1(\overline{L})^d)\subset H.\]
Par \cite{Hironaka}, il existe  un morphisme propre $\pi:\widetilde{X}\rightarrow X$ avec
$\widetilde{X}$ est une vari\'et\'e non-singuli\`ere et $E=\pi^{-1}(H)$ un diviseur \`a croisements
normaux. On a $\pi^\ast \overline{L}$ est admissible  et
\[
 \mathrm{\mathrm{Supp}}(c_1(\pi^\ast \overline{L})^d)\subset E.
\]
V\'erifions ce dernier point, d'apr\`es \cite{Hironaka}, $\pi_{\widetilde{X}\setminus E}$ est un
isomorphisme en $\widetilde{X}\setminus E$ et $X\setminus H$, donc $
\bigl(c_1(\pi^\ast\overline{L})^d \bigr)_{|_{\widetilde{X}\setminus
E}}=\pi^\ast\bigl(c_1(\overline{L})^d \bigr)_{|_{X\setminus H}} \bigr)=0$, car
$\mathrm{\mathrm{Supp}}(c_1(\overline{L})^d)\subset H$ par hypoth\`ese. En d'autres termes, on peut supposer
que $H$ est un diviseur \`a croisements normaux. \\


On munit $\mathcal{O}(H)$ d'une m\'etrique hermitienne $\cl$, qu'on note par $\vc_H$. Soit $x\in H$, il existe un voisinage ouvert $V$  de $x$,  un syst\`eme de coordonn\'ees locales $\{z_1,\ldots,z_d\}$ centr\'e en $x$ et un entier $k$ tels que
\[
 H\cap V=\{z_1\cdots z_k=0\}=\cup_{j=1}^k\{z_j=0\}.
\]

Soit $j=1,\ldots,k$ et $\eps>0$. On consid\`ere   $H_{j,\eps}$, le sous ensemble ouvert de $V$ donn\'e par $H_{j,\eps}=\{|z_j|<\eps\}$.
Soit $\rho_{j,\eps}$ une fonction \`a valeurs dans $[0,1]$, $\cl$ sur $X$ et  \`a support dans $V$,
nulle sur $V\setminus
H_{j,2\eps}$ et \'egale
\`a 1 sur $H_{j,\eps}$.

Comme le support de $c_1(\overline{L})^d$ est inclus dans $H$, alors pour
$\eps$ assez petit, le courant $\bigl(1-\rho_{j,\eps}\bigr) c_1(\overline{L})^d$ est nul sur $X$. Et  on a,
{\allowdisplaybreaks
\begin{align*}
\int_X \log \|z_j\|_{{}_{H}} c_1(\overline{L})^d=& \int_X \Bigl(\rho_{j,\eps}(z)\log \|z_j\|_{{}_{H}}+(1-\rho_{j,\eps}(z))\log\|z_j\|_{{}_{H}}\Bigr) c_1(\overline{L})^d\\
&=\int_X \rho_{j,\eps} (z)\log\|z_j\|_{{}_{H}}c_1(\overline{L})^d\\
&=\int_{H_{j,2\eps}} \rho_{j,\eps}(z)\log \|z_j\|_{{}_{H}}c_1(\overline{L})^d\\
&\leq \log(2\eps)\int_{H_{j,2\eps}}\rho_{j,\eps}(z)c_1(\overline{L})^d+\log C\int_{H_{j,2\eps}}\rho_{j,\eps}(z)c_1(\overline{L})^d,
\end{align*}}
(car localement, il existe $C>0$ tel que $\|z_j\|_{{}_H}\leq C|z_j|$).\\

Par construction de $\rho_{j,\eps}$, on v\'erifie que pour tout $\eps>0$:
\begin{align*}
\int_{H_{j,\eps}}c_1(\overline{L})^d\leq
\int_{H_{j,2\eps}}\rho_{j,\eps}(z)c_1(\overline{L})^d\leq
\int_{H_{j,2\eps}}c_1(\overline{L})^d,
\end{align*}

on en d\'eduit que  la suite
$\Bigl(\int_{H_{j,\eps}}\rho_{j,\eps}c_1(\overline{L})^d\Bigr)_{0<\eps\ll 1} $ converge vers une limite finie $l$ lorsque $\eps$ tend vers z\'ero.  Comme
\[
 \int_X \log \|z_j\|_{{}_H} c_1\bigl(\overline{L}\bigr)^d,
\]
est finie, cela r\'esulte de la th\'eorie de  Bedford et Taylor  (voir \cite{Bedford}), on doit avoir n\'ecessairement $l=0$.
Par
cons\'equent
\[
 \int_{H_{j,\eps}}c_1(\overline{L})^d \xrightarrow[\eps\lra 0]{}0.
\]
 En posant $H_\eps=\cup_{j=1}^k H_{j,\eps}$, on a alors
\[
 \int_{H_{\eps}}c_1(\overline{L})^d \xrightarrow[\eps\lra 0]{}0.
\]
On peut trouver un recouvrement finie par des ouverts  $(V_\al)_{\al\in J}$ comme $V_\al$, et on pose comme plus
$H_{\al,\eps}$, alors $H\subset \cup_{\al\in J} H_{\al,\eps}$.

Soit $\rho$ une fonction continue positive sur $X$, on a pour tout $0<\eps\ll 1$
\begin{align*}
0\leq  \int_X\rho\, c_1(\overline{L})^d&\leq \int_{X\setminus \overline{H}_\eps}\rho\,
c_1(\overline{L})^d+\int_{\cup_{\al\in J}H_{\al,2\eps}}\rho\, c_1(\overline{L})^d\\
&\leq \int_{X\setminus \overline{H}_\eps}\rho\, c_1(\overline{L})^d+\sum_{\al\in
J}\int_{H_{\al,2\eps}}\rho\, c_1(\overline{L})^d\\
&=0+\sum_{\al\in
J}\int_{H_{\al,2\eps}}\rho\, c_1(\overline{L})^d\quad \bigl(\text{car}\;
c_1(\overline{L})^d_{|_{X\setminus \overline{H}_\eps}}=0\bigr)\\
&\leq \|\rho\|_{\sup}\sum_{\al\in J}\int_{H_{\al,2\eps}} c_1(\overline{L})^d.
\end{align*}
Ce qui donne par passage \`a la limite: $\int_X\rho\, c_1(\overline{L})^d=0$. Ce qui contredit les hypoth\`eses du 
th\'eor\`eme. 
On conclut  que si $\overline{L}$ est un fibr\'e en droites admissible 
tel que  $c_1(L)^d\neq 0$  alors le support de
$c_1(\overline{L})^d$ est Zariski-dense dans $X$.\\

Par le m\^eme raisonnement on montre que $\int_X\rho\, c_1(\overline{L})^d= \int_{X\setminus
H}\rho\, c_1(\overline{L})^d$, en effet, on a
{\allowdisplaybreaks
\begin{align*}
0\leq  \int_X\rho\, c_1(\overline{L})^d&\leq \int_{X\setminus \overline{H}_\eps}\rho\, c_1(\overline{L})^d+\int_{\cup_{\al\in J}H_{\al,2\eps}}\rho\, c_1(\overline{L})^d\\
&\leq \int_{X\setminus H}\rho\, c_1(\overline{L})^d+\sum_{\al\in
J}\int_{H_{\al,2\eps}}\rho\, c_1(\overline{L})^d\\
&\leq \int_{X\setminus H}\rho\, c_1(\overline{L})^d+\sum_{\al\in
J}\int_{H_{\al,2\eps}}\rho\, c_1(\overline{L})^d\\
&\leq \int_{X\setminus H}\rho\, c_1(\overline{L})^d+\|\rho\|_{\sup}\sum_{\al\in J}\int_{H_{\al,2\eps}} c_1(\overline{L})^d,
\end{align*}}
 ce qui donne que $\int_X\rho\, c_1(\overline{L})^d\leq \int_{X\setminus
H}\rho\, c_1(\overline{L})^d$. Par suite $\int_X\rho\, c_1(\overline{L})^d= \int_{X\setminus
H}\rho\, c_1(\overline{L})^d$ pour toute fonction continue positive $\rho$, sur $X$. Lorsque
$\rho$ est une fonction continue quelconque, 
on \'ecrit $\rho=\rho_+-\rho_-$, o\`u
$\rho_+=\max(\rho,0)$ et $\rho_-=\max(-\rho,0)$. Et on d\'eduit que
\[
 \int_X\rho\, c_1(\overline{L})^d= \int_{X\setminus
H}\rho\, c_1(\overline{L})^d.
\]

\end{proof}

\section{Appendice}\label{x100}

Dans cette section, on rappelle quelques r\'esultats donnant le lien
entre certaines propri\'et\'es combinatorials des matrices et
les propri\'et\'es alg\'ebriques d'id\'eaux binomiaux. On expliquera
apr\`es comment produire des exemples de vari\'et\'es toriques
v\'erifiant la condition $\mathscr{A}$.

 Dans 
\cite{Herzog}, on consid\`ere les alg\`ebres de la forme
$R[S]$ avec $R$ un anneau commutatif unitaire et $S$ un semigroupe
commutatif de type fini et on 
 \'etudie les conditions sous lesquelles
ces alg\`ebres ont une repr\'esentation finie, 
une intersection compl\`ete... (voir aussi \cite{Delorme}).
Dans \cite{Eisenbud}, on \'etudie les id\'eaux binomiaux dans 
l'anneau des poly\^omes de Laurent \`a coefficients dans un 
corps. Un des r\'esultats principaux de \cite{Eisenbud}  donne
une condition suffisante pour qu'un id\'eal binomial $I=(f_1,\ldots,
f_r)$, o\`u $f_1,\ldots,f_r$ sont des bin\^omes, d\'efinit
une suite r\'eguli\`ere dans cet anneau.
Cette condition suffisante
 stipule que les exposants de $f_1,\ldots,f_r$
forment un syst\`eme lin\'eairement ind\'ependant. 
Cette condition n'est plus suffisante si $I$ est un id\'eal de
 l'anneau  
des polyn\^omes \`a coefficients dans $\Z$, voir 
\cite[example 2.1]{Mixed}. Dans \cite{Mixed}, les auteurs
donnent une condition  pour que $I$ soit premier et
l'alg\`ebre $\Z[S]$ soit une intersection compl\`ete, o\`u $S$ est un semigroupe associ\'e \`a la
matrice $M$ d\'efinie par les exposants de $f_1,\ldots,f_r$. Pour cela, ils
montrent qu'il suffit que $M$ v\'erifie certaines propri\'et\'es 
combinatorials, voir \cite[theorem 2.9, Corollary 2.10]{Mixed}.

Soient $k$ et $n$ deux entiers positifs avec $k\leq n$. On note par
$\mathcal{M}(k\times (n+1),\Z)$ l'ensemble des matrices de taille
$k\times (n+1)$ \`a coefficients dans $\Z$. Soit $M\in \mathcal{M}(k
\times (n+1),\Z) $.
On dit que $M$ est \textit{mixte} si chaque ligne contient  deux
coefficients de signe diff\'erent. $M$ est dite \textit{dominante} 
si $M$ ne contient pas une sous-matrice mixte de taille $k\times k$. 
On appelle \textit{contenu} de $M$ et on le note par $\mathrm{cont}(M)$,
le pgcd des mineurs de taille $k\times k$ de $M$. 

 Soit $\Omega$ un anneau commutatif unitaire. Soit $v=(v_0,\ldots,v_n)\in 
 \Z^{n+1}$. On pose 
$v_+:=(\max(0,v_0),\ldots,\max(0,v_n))$ et $v_-:=(\max(0,-v_0),\ldots,\max(0,-v_n))$. Si $v\in \N^{n+1}$, on note par 
$T^v$ le polyn\^ome de $\Omega[T_0,\ldots,T_n]$ d\'efini comme suit
$T^v=T_0^{v_0}\cdots T_n^{v_n}$ et on pose 
$L_u:=T^{u_+}-T^{u_-}$ pour tout $u\in \Z^{n+1}$.

 On consid\`ere 
$u_1,\ldots,u_k$, $k$ vecteurs de $\Z^{n+1}$ et on pose
\begin{equation}\label{matrice}
M=\begin{bmatrix}
u_{10}&u_{11} & u_{12} & \ldots
& u_{1n} \\
u_{20}&u_{21} & u_{22} & \ldots
& u_{2,n+1} \\
\vdots & \vdots & \ddots
& \vdots \\
u_{k,0}&u_{k,1} & u_{k,2} & \ldots
& u_{k,n+1}
\end{bmatrix},
\end{equation}
 $L_i:=T^{u_{+,i}}-T^{u_{-,i}}$ pour 
 $i=1,\ldots,k$, $I_M:=(L_1,\ldots,
 L_k)$ l'id\'eal de $\Omega[T_0,\ldots,T_n]$ engendr\'e par 
 $L_1,\ldots,L_k$, 
 $I^\ast_M:=(L_u| u\in \sum_{i=1}^k\Z u_i)$ et
 $\widehat{I}_M:=(L_u| u\in \bigl(\sum_{i=1}^k \Q u_i\bigr)\cap \Z^{n+1})$. Clairement, on a
 \[
 I_M\subset I_M^\ast\subset \widehat{I}_M.
 \]

 \begin{proposition}\label{w1}
Suppose que $\Omega=\Z$.
 Soit $M\in \mathcal{M}(k\times (n+1),\Z)$ une matrice mixte, dominante telle que les lignes sont lin\'eairement ind\'ependantes. Alors
$I_M=I_M^\ast$. R\'eciproquement, si les lignes de $M$
sont lin\'eairement ind\'ependantes et $M$ est mixte alors 
le fait que $I_M=I_M^\ast$ implique que 
$M$ est dominante.
\end{proposition}
 \begin{proof}
 voir \cite[theorem 2.9]{Mixed}.
.
 \end{proof}
 
 On d\'efinit $G_M:=\Z^{n+1}/(u_1,\ldots,u_k)$ et 
 $S_M$ le semi-groupe de $G_M$ engendr\'e par $e_0,\ldots,e_n$, 
 la base standard de $\Z^{n+1}$. On note par 
 $\rho^\ast$ l'application surjective suivante,
 \[
 \rho^\ast:\N^{n+1}\twoheadrightarrow S_M,\quad v=(v_0,\ldots,v_n)\mapsto 
 \sum_{i=0}^n v_ie_i.
 \]
Cette application induit le 
morphisme surjectif d'alg\`ebres suivant
\[
\Omega[\rho^\ast]:\Omega[T_0,\ldots,T_n]\twoheadrightarrow \Omega[S_M],\quad
T_i\mapsto 1_{\rho^\ast(e_i)}.
\]
On pose $I_{S_M}:=\ker (\Omega[\rho^\ast])$. 
D'apr\`es
Herzog, on peut d\'efinir une $S_M$-graduation 
sur $\Omega[T_0,\ldots,T
_n]$ comme suit: $F\in \Omega[T_0,\ldots,T_n]$ est dit
homog\`ene  de degr\'e $s\in S_M$, si $F=\sum_{v}\al_v X^v$ 
avec
$\rho^\ast(v)=s$ pour tout $s$ tels que $\al_v\neq 0$, 
(voir \cite[p.177]{Herzog}).

On montre que 
\[
I_{S_M}=I_M^\ast.
\]
voir \cite[lemma 4.35]{Binomial}. Rappelons la preuve;   Soit
$u\in \sum_{i=1}^k\Z u_i$, on  $\Omega[\rho^\ast](L_u)
=1_{\rho^\ast(u_+)}-1_{\rho^{\ast}(u_-)}=0$, puisque
$u_+-u_-=u\in (u_1,\ldots,u_k)$, donc
$I_M^\ast\subset I_{S_M}$. R\'eciproquement, soit
$L=\sum_{i}^m\al_i X^{v_i}\in I_{S_M}$ homog\`ene de degr\'e 
$s\in S_M$, donc 
$(\sum_i \al_i) 1_s=0 $. Par suite,
\[
L=\sum_i \al_i (X^{v_i}-X^{v_m})=\sum_i \al_i X^{\inf(v_i,v_m)}
(X^{(v_i-v_m)_+}-X^{(v_i-v_m)_-})
\]
et $v_i-v_m\in \sum_{i=1}^k \Z u_i$ (car rappelons que $\rho
^\ast(v_i)=s$ pour tout $i$.). On conclut que
\[
I_{S_M}=I_M^\ast.
\]

Lorsque $\mathrm{cont}(M)=1$, alors
on montre que $G_M\simeq \Z^{n+1-k}$. En particulier, on a dans ce cas
$I_{S_M}$ est premier. En effet, on a $\Omega[S_M]\subset \Omega[G_M]
\simeq \Omega[T_0^+,T_0^-,\ldots,T_{n-k}^+,T_{n-k}^-]$ qui est int\`egre.

\begin{proposition}\label{w2}
Suppose que $\Omega=\Z$. On a
$\Z[S_M]$ est intersection compl\`ete 
si et seulement $M$ est dominante avec contenu $1$.
\end{proposition}
\begin{proof}
voir \cite[corollary 2.10]{Mixed}.
\end{proof}
\begin{proposition}
Soient $M$ une matrice mixte $k\times (n+1)$ et $\Omega$ un anneau
commutatif unitaire et int\`egre. Les assertions suivantes sont \'equivalentes:
\begin{enumerate}
\item $M$ est mixte et dominante avec $\mathrm{cont}(M)=1$.
\item $I_M\subset \Omega[T_0,\ldots,T_n]$ est premier et 
$u_1,\ldots,u_k$ sont lin\'eairement ind\'ependants. 
\item $I_M\subset K[T_0,\ldots,T_n]$ est premier d'hauteur $k$, o\`u 
$K$ 
est un corps.
\end{enumerate}  
\end{proposition}
\begin{proof}
\cite[proposition 4.47]{Binomial}.  Notons que 
l'implication $1)\implies 3)$ peut se d\'eduire des 
propositions \eqref{w1} et \eqref{w2} et en notant que 
si ces r\'esultats sont vrais pour un anneau $\Omega_0$ alors
ils restent valables pour tout anneau.

\end{proof}

Soit $M$ comme dans \eqref{matrice}, et on suppose que 
$u_{i0}=-\sum_{j=1}^nu_{ij}$ pour tout $i=1,\ldots,k$. Donc, les
polyn\^omes $L_1,\ldots,L_k$ sont homog\`enes, et $I_M$ est un id\'eal 
homog\`ene (par rapport \`a la graduation usuelle) d\'efinissant ainsi une sous-vari\'et\'e projective. Soient $K$ un corps et
$\gamma=(\gamma_0,\ldots,\gamma_n)\in (K^\ast)^{n+1}$ avec 
$\gamma_0=1$. On note par
$[\gamma]$ l'isomorphisme d'alg\`ebres suivant
\[
[\gamma]:K[T_0,\ldots,T_n]\rightarrow K[T_0,\ldots,T_n]\quad
T_i\mapsto \gamma_i T_i.
\]
On a clairement, $I_M$ est premier dans $K[T_0,\ldots,T_n]$ 
si et seulement si $[\gamma
]^\ast I_M$ l'est aussi.  

\subsection{Une application }

Dans ce paragraphe on utilise les notations de la section
\eqref{x200}. Soient $d\leq n$, $\beta\in (\overline{\Q}^\ast)^{n+1}$ avec $\beta_0=1$, $K$ et $R_1,\ldots,R_{n-d}$ comme dans 
la section \eqref{x200}. 
On ais\'ement v\'erifie que $[\beta]^\ast (R_1,\ldots,R_{n-d})=(L_1,\ldots,
L_{n-d})$ dans $K[T_0,\ldots,T_n]$. Par suite,
si
$M$ est mixte, dominante et $\mathrm{cont}(M)=1$, alors d'apr\`es ce
qui pr\'ec\`ede, l'id\'eal $(R_1,\ldots,R_{n-d})$ est premier
dans $K[T_0,\ldots,T_n]$. Cela nous motive \`a introduire la
d\'efinition suivante

\begin{definition}\label{mmd}
Soit $M$ comme dans \eqref{matrice}, avec
$u_{i0}=-\sum_{j=1}^n u_{ij}$ pour tout $i=1,\ldots,k$. On dit que 
$M$ v\'erifie l'hypoth\`ese $(\mathcal{I})$ s'il existe $(M_i)_{i=1,
\ldots,k}$ une suite 
de sous-matrices de $M$ de taille $i\times (n+1)$ respectivement, telles
que $M_i$ est une sous-matrice de $M_{i+1}$ pour tout $i=1,\ldots,k-1$ 
et $M_i$ est mixte, dominante et $\mathrm{cont}(M_i)=1$.
 
\end{definition}

\begin{remarque}
Si $M\in \mathcal{M}(k\times (n+1),\Z)$ est mixte (resp. v\'erifie $\mathrm{cont}(M)=1$) alors toute sous-matrice de $M$ de taille
$i\times (n+1)$ est mixte (resp. a un contenu \'egal \`a 1).
\end{remarque}

\begin{example}
Soient $c_1,c_2,c\in \N_{\geq 1}$, 
\[
M_0=\begin{bmatrix}
0& c_1 &1-c_1 &-1\\
-c_2-1& c_2& 1 &0,
\end{bmatrix}
\]
et

\[
M_1=\begin{bmatrix}
c-1&-c&1&0&0\\
0&-1&0&1&0\\
0&-1&0&0&1\\
\end{bmatrix}
\]
Alors $M_0$ et $M_1$ v\'erifient les hypoth\`eses de la d\'efinition 
\eqref{mmd}. On v\'erifie que $M_0$ d\'efinit la courbe $X_{\A_0,1}$
dans $\p^3$ avec $\A_0=\{1,-c_2,c_1-c_2+c_1c_2  \}$.
\end{example}

\begin{proposition}\label{modeleintegre}
Soit $M\in \mathcal{M}(n-d\times (n+1),\Z)$ qui v\'erifie l'hypoth\`ese
$(\mathcal{I})$. Soit $\beta=(1,\beta_1,\ldots,\beta_n)\in (\overline{\Q}^\ast)^{n+1}$, alors 
$X_{\A,\beta}$ v\'erifie la condition $\mathscr{A}$ (voir
D\'efinition \eqref{condA}). 
\end{proposition}

\begin{proof} Comme $M$ v\'erifie l'hypoth\`ese $(\mathcal{I})$, alors
en particulier $M$ est dominante, mixte avec $\mathrm{cont}(M)=1$, donc
l'id\'eal $(R_1,\ldots,R_{n-d})$ est premier dans $K[T_0,\ldots,T_n]$, o\`u 
$K=\Q(\beta_1,\ldots,\beta_n)$. Quitte \`a r\'eordonner
les indices, on obtient par induction que $(R_i,\ldots,R_{n-d})$ est premier pour
tout $i=1,\ldots,n-d$.
\end{proof}

\begin{remarque}
Signalons qu'on dispose d'un algorithme en temps polyn\^omial
permettant de reconna\^itre si une matrice est mixte et dominante, voir
 \cite[p. 198]{Mixed2}.
\end{remarque}

\bibliographystyle{plain}
\bibliography{biblio}

\end{document}